\documentclass[11pt]{amsart}
\usepackage{amssymb,amsgen,amsbsy,amsopn,amsfonts,graphicx}
\usepackage{amsmath}
\usepackage[dvips]{color}
\usepackage{multicol}
\usepackage{mathrsfs}
\usepackage{nccmath}
\usepackage{enumerate}
\usepackage{url}
\usepackage{amsthm}
\usepackage{color,hyperref}
\newtheorem{lemma}{Lemma}[section]
\newtheorem{mainthm}{Main Theorem}
\newtheorem*{mthm}{Main Theorem}
\newtheorem{theorem}[lemma]{Theorem}
\newtheorem{prop}[lemma]{Proposition}
\newtheorem{cor}[lemma]{Corollary}
\newtheorem{rem}[lemma]{Remark}

\theoremstyle{definition}

\numberwithin{equation}{section}
\numberwithin{lemma}{section}
\newcommand{\p}{\partial}
\def\R{{\mathbb R}}
\def\S{{\mathbb S}}
\def\H{{\mathbb H}}
\newcommand{\al}{\alpha}
\newcommand{\CJ}{\mathcal{J}}

\begin{document}

\title[Hyperbolic spin-field system: 
standing waves]{Equivariant and self-similar
standing waves for a Hamiltonian hyperbolic-hyperbolic spin-field system.}
\author[Nan Lu]{Nan Lu$^1$}
\address{$^1$Department of Mathematics \\
University of Massachusetts\\ 710 N. Pleasant Street, Amherst MA 01003}
\email{nanlu@math.umass.edu}
\author[Nahmod]{Andrea R. Nahmod$^2$}
\address{$^2$Department of Mathematics \\
University of Massachusetts\\ 710 N. Pleasant Street, Amherst MA 01003}
\email{nahmod@math.umass.edu}
\thanks{$^2$ The second author is funded in part by NSF DMS 1201443. This work was also partially supported by a grant from the Simons Foundation ($\#$ 267037 to Andrea R. Nahmod as 2013--2014 Fellow)}
\author[Zeng]{Chongchun Zeng$^{3}$}
\address{$^{3}$ School of Mathematics\\
Georgia Institute of Technology\\ 686 Cherry Street, Atlanta, GA 30332}
\email{zengch@math.gatech.edu}
\thanks{$^{3}$ The third author is funded in part by NSF
DMS  1362507 and a part of the work was carried out while visiting IMA, University of Minnesota, in 2012--2013.}
\date{}
\begin{abstract}
In this paper we study the existence of special symmetric solutions to a {\it Hamiltonian} hyperbolic-hyperbolic coupled spin-field system, where the spins are  maps from $\mathbb R^{2+1}$ into the sphere $\S^2$  or the pseudo-sphere $\H^2$. This model was introduced by Martina {\it et al.} in \cite{Martina} from the hyperbolic-hyperbolic generalized Ishimori systems.
Relying on the hyperbolic coordinates introduced in \cite{KNZ11}, we prove the existence of 
equivariant standing waves both in regular hyperbolic coordinates as well as in similarity variables,  and describe their asymptotic behavior. 
\end{abstract}
\maketitle

\section{Introduction}\label{In}
In this paper we study hyperbolically equivariant standing wave solutions in standard as well as in self-similar variables  for the following Hamiltonian hyperbolic-hyperbolic spin-field system (a modified Ishimori system),
\begin{equation}\label{eq1.1}
\left\{\begin{aligned} & u_t=u\wedge_\pm \square_{xy} u+u_y\phi_x-u_x\phi_y,\\
&\phi_{xx}-\phi_{yy}=-2\langle u, u_x\wedge_\pm u_y\rangle_\pm,
\end{aligned}\right.
\end{equation} for maps $u$ from $\R^{2+1}$ into either the unit sphere $\mathbb S^2$ or the pseudo sphere $\mathbb H^2$ which can be isometrically and equivariantly embedded  into $\R^3$ equipped with the standard Euclidean or the Lorentz metric $\eta_\pm:= diag(\pm 1,1,1)$; and $\phi=\phi(x, y, t)$,  a real scalar field. 
As usual, we denote by $\square_{xy}\,:=\, \p_{xx}-\p_{yy},$ while we denote by 
$\langle u, v\rangle_\pm:=  u\cdot (\eta_\pm v),$ and by $u\wedge_\pm v :=  \eta_\pm(u\wedge v) $
with $\cdot$ and $\wedge$ being the usual dot and wedge product in $\R^3$.

\smallskip
 
Spin-field models of this type were proposed by physicists investigating nonlinear excitations arising in a broad range of physical backgrounds including ferromagnetism and quark confinement. In this context, a basic role is played by the nonlinear $\sigma$-model in $2+1$ dimensions. As a nonrelativistic dynamical version of this model, the Ishimori system, first in the elliptic-hyperbolic and the hyperbolic-elliptic forms, was proposed by Y. Ishimori \cite{Ish}, seeking to show that the dynamics of topological vortices need not generally be non-integrable in analogy with the $2$-dimensional classical continuous isotropic Heisenberg spin chain. Here the ellipticity and hyperbolicity refer to certain signs in the spin and the field equations. Those Ishimori systems proposed in \cite{Ish} possess Lax pairs and are thus integrable. In particular, they were proved to be gauge equivalent to the corresponding Davey-Stewartson systems and have been studied by inverse scattering methods (\cite{KONOMATa, KONOMATb, SUNG}). Later all four possible sign combinations (elliptic-elliptic, elliptic-hyperbolic, hyperbolic-elliptic, and hyperbolic-hyperbolic) in the Ishimori system have been studied and some nonlinear excitations ({\it solitons, vortexlike, doubly periodic solutions}) with integer topological charges have been found  \cite{DK91, Ish, LMS90, LMS92}.  However to the best of our knowledge,  it is not clear whether all the Ishimori systems have Hamiltonian structures. 

Martina {\it et. al.} \cite{Martina} proposed a modification of the Ishimori systems allowing for a Hamiltonian formulation and of which \eqref{eq1.1} corresponds to the hyperbolic-hyperbolic case.  Interestingly,  even though the modification is not known to be integrable or not, it can be put into a Hirota form which leads to an infinite dimensional symmetry algebra of the Kac-Moody type with a loop-algebra structure -- a property shared by other integrable nonlinear field equations in $2+1$ dimensions (eg. KP equations, Davey Stewartson system).  Some simple nonlinear excitations are also found with explicit algebraic forms  in \cite{Martina}. Some of those nonlinear excitations (the {\it helical-type, rotonlike}, and the {\it radial static }configurations) are essentially planar, meaning the spin ranges in a 1-dimensional curve on the target manifold $\mathbb{S}^2$ or $\mathbb H^2$. The spin of the {\it meronlike} configuration covers a hemisphere of $\mathbb S^2$. Moreover, a {\it domain wall} configuration  is also given which are traveling waves in $y$,  equivariant in $x$ with rather simple algebraic form. 

Since the 80's, well-posedness and regularity analysis has been carried out for a class of spin-field systems including generalized Ishimori systems. Local and global well-posedness results under various regularity, size and/or decay conditions on the data can be found in the literature (e.g. \cite{SOU, LiPo, HAYSAUa, HAYSAUb, HAYNAU, KPV, CHPS, KNH, BIK} and references therein). In particular, Hayashi and Saut studied the Cauchy problem of this type of systems for maps into $\S^2$ and obtained well-posedness results for small smooth solutions with/without exponential decay assumption on initial data in \cite{HAYSAUa, HAYSAUb}, which applies to system \eqref{eq1.1} in the case of target manifold $\mathbb S^2$. Moreover, the spin-field systems in the elliptic-elliptic form are intimately related to the Schr\"odinger map system, which has received increased attention in the last few years (e.g. \cite{SSB, CSU, NSU, McG, KNH, NSVZ, B, BIKT, BIK} and references therein). 
 
We are in particular interested in finding invariant structures playing important roles in the dynamics of \eqref{eq1.1} and which do not necessarily have simple algebraic forms.  In the present paper,  we focus on standing waves. These are a type of relative equilibria, whose temporal evolution is either trivial or coincides with some basic invariant transformations of the system leaving the spatial profile of the solutions essentially unchanged. For typical dispersive PDEs like nonlinear Schr\"odinger equations or wave maps, they play crucial roles in the analysis of the asymptotic dynamics of the corresponding systems and often turn out to be spatially radial or equivariant with respect to Euclidean rotations. 



 



The fact that one may succeed in seeking radial or radially equivariant standing waves for nonlinear Schr\"odinger equations, wave maps {\it etc.} is mostly due to the invariance of these equations under the Euclidean rotations in both the target spaces and the spatial variables. In contrast, on the one hand  system \eqref{eq1.1} is invariant under the composition with any isometry of the target surface $\mathbb S^2$ or $\mathbb H^2$, referred to as the conformal invariance in \cite{Martina}. On the other hand, \eqref{eq1.1} is invariant under the hyperbolic (instead of the Euclidean) rotations in the $xy$-plane. These are the most basic invariances of the hyperbolic spin-field system \eqref{eq1.1}. In view of our results in \cite{KNZ11}, we thus seek the existence and the spatially asymptotic behavior of standing wave solutions to \eqref{eq1.1}  equivariant under these symmetries. 



System \eqref{eq1.1} has more complex nonlinear behaviors than that of the purely power one of the scalar cubic hyperbolic nonlinear Schr\"odinger equation studied in \cite{KNZ11}. In analyzing the problem of standing waves solutions with the above symmetries, we adapt the geometric framework of \cite{GSZ} and \cite{DTZ}  in the context of Schr\"odinger maps,  to reduce \eqref{eq1.1} to an ODE system of the independent variable $a= \sqrt{|x^2-y^2|}$, which is one component of the hyperbolic coordinates  $(a,\alpha)$ introduced in \cite{KNZ11} (see section~\ref{S:coord} for the exact change of
coordinates). The solutions of this ODE system lead to equivariant and self similar standing waves solutions which are  weak solutions satisfying suitable compatibility conditions (see
section~\ref{weak} for discussions on weak solutions).  We also study the asymptotic behavior at spatial infinity of the solutions we construct. In fact, we observe rather different asymptotic properties between the different target manifolds $\mathbb S^2$ and $\mathbb H^2$. When the problem is considered in similarity variables, equivariant standing waves behave even more differently at spatial infinity. 


The maps of our standing wave solutions- just as the hyperbolically radial standing waves found in \cite{KNZ11}-  do not have finite $L^2$ gradient since integration in $d \alpha$ becomes infinite{\footnote{Indeed note that the measure $dx dy = |a| da d\alpha$ with $\alpha \in \R$; and that the `equivariant'  case singles out a mode in a {\it continuous} family of ``rotation numbers".  In contrast to the elliptic case where rotation numbers are only integers, solutions with hyperbolic symmetry thus should be viewed as a singular case.}}  for functions depending solely on $a$. Just as it was the case in \cite{KNZ11}, and for the purposes of this paper, this situation is not too objectionable. Some important elementary waves (eg. plane waves, kinks, etc.) have infinite $L^2$ norm. Furthermore, since the hyperbolic Laplacian $\square_{xy} = \partial_{xx}- \partial_{yy}$ has characteristics and does not correspond to a positive energy, it is reasonable, as an initial step, to relax the finite $L^2$ gradient 
requirement and look for hyperbolically equivariant continuous standing waves.  

Our strategy is to first parameterize the target manifold
by polar coordinates $(s,\beta)$ with tangent vectors denoted by
$(\p_s,\p_\beta)$ and the domain $xy$-plane by hyperbolic coordinates $(a, \alpha)$ (see section~\ref{S:coord}). We express the unknown map $u$ in terms of $(s,\beta)$ to obtain a singular ordinary
differential equation system. 
With the unit sphere $\mathbb S^2$ and the pseudo sphere $\mathbb H^2$ isometrically and equivariantly embedded into $\R^3$ equipped with the standard Euclidean or the Lorentz metric $\eta_\pm:= diag(\pm 1,1,1)$ as $\mathbb S^2 = \{u_0^2 + u_1^2 + u_2^2=1\}$ and $\mathbb H^2 =\{ u_0^2 - u_1^2 - u_2^2 =1 \}$, our combined results can be generally stated as follows:

\begin{mthm} \label{thm1}
The system \eqref{eq1.1} admits equivariant standing wave weak solutions, which are smooth off the cross $\{|x|=|y|\}$. In each of the four quadrants separated by $\{|x|=|y|\}$, these solutions take the form 
\[\begin{aligned}
u (t,x,y)=&\big(u_{0}, u_{1}, u_{2}\big)\\
=&\big(\Gamma_s(s(a)), \Gamma(s (a)) \cos{(\beta (a)+ k\alpha+\mu t)}, 
\Gamma(s(a))\sin{(\beta(a)+ k \alpha+\mu t)}\big),\\
\phi (t,x,y)=&c\alpha + b \log a - \int_0^{a}\frac{2k}{a'}F(s(a'))\ da',
\end{aligned}\]
where $\Gamma(s) = \sin s$ for $\mathbb S^2$ and  $\Gamma(s) = \sinh s$ for $\mathbb H^2$,
\[
a = \sqrt{|x^2 - y^2|}, \quad \alpha = \frac 12 \log \frac {|x+y|}{|x-y|}, 
\]
$b$ is any constant satisfying $-4(k^2 + bk) > c^2$ and $F(s) = \int_0^s \Gamma(s') ds'$.  
Moreover, $s (0)=0$ and $s (a)$ and $\beta (a)$ are smooth for $a>0$ and are locally $C^\kappa$ 
for $a\ge 0$ with the exponent $\kappa = \sqrt{-(k^2+bk)-\frac{c^2}{4}}$. As $a\to \infty$, while in the $\mathbb H^2$ case $s(a) = O(a^{-\frac 12})$, in the $\mathbb S^2$ case $s(a) = \pi + O(a^{-\frac 12})$ in one pair of quadrants (determined by $|x|>|y|$ or $|x|<|y|$) and $s(a) = O(a^{-\frac 12})$ in the the other pair of quadrants.

The system also admits  equivariant standing wave weak solutions in the similarity variables $(t^{-\frac 12} x, t^{-\frac12}y)$, whose spins are H\"older continuous away from $t=0$ and smooth off the cross $\{|x|=|y|\}$ and $t=0$. For fixed $t\ne 0$, as $|x^2 -y^2| \to \infty$, the spin converges to some circle 
on the target surface at the rate of $|x^2 -y^2|^{-1}$. 
\end{mthm}

While the full statement of the above results and a more detailed description of the asymptotic behavior will be given in Section \ref{eq} and \ref{self-similar}, we make some brief comments here. Firstly even though the above field $\phi$ has logarithmic singularities at $|x|=|y|$, the spin $u$ is H\"older everywhere (actually enjoys stronger regularity compared to most of the nonlinear excitations found in \cite{Martina}) and $(u, \phi)$ qualify as weak solutions in the distribution sense defined in Section \ref{weak}. Secondly, while the image of the spin $u(t, x, y)$ along any hyperbola $x^2 - y^2 =\pm a^2$ is a circle centered at the rotation center, its image along any ray emitting from the origin is not a very simple curve on the target surface (definitely not a geodesic), which somewhat illustrates the complexity of these solutions. Thirdly, in the case of the target surface $\mathbb H^2$, the spin $u$ as a mapping from the $xy$-plane to $\mathbb H^2$ has trivial topological degree, as $s(a) \to $ the rotation center as $a \to \infty$. This is drastically different in the $\mathbb S^2$ case.While, in the one pair of quadrants (determined by $|x|>|y|$ or $|x|<|y|$) the spin $u$ still has trivial topological degree, it is topologically nontrivial in the the other pair of quadrants. In each of the latter quadrant, $u$ maps $|x|=|y|$ to the rotation center $(1,0,0) \in \mathbb S^2$ and $u\to (-1, 0, 0)$ as $|x^2 - y^2| \to \infty$ and thus the topological degree is actually $\infty$ due to the rotation in $\alpha$. As it demonstrates a transition of the spin $u$ from $(1,0,0)$ to $(-1, 0,0)$, these standing waves may be viewed as a kind of  the domain wall type.  Finally in the case of a standing wave in the similarity variables, as $|x^2-y^2| \to \infty$, the image of the spin converges to a circle most likely not the rotation center, which can be confirmed at least for some choices of parameters.



This work is part of a program, started with \cite{KNZ11}, devoted to the study of fundamental open questions for nonelliptic nonlinear Schr\"odinger equations.  Our long term goal is to help bring the analysis of this relatively untrodded area onto an equal footing with its elliptic counterpart. As standing waves have played crucial roles in the analysis of dispersive PDEs with elliptic differential operators, we expect that the results as well as the techniques in these papers will prove instrumental in the study of the dynamics of non-stationary equivariant (as well as general) solutions under \eqref{eq1.1} or other nonelliptic nonlinear dispersive PDEs like the hyperbolic nonlinear Schr\"odinger equation {\it etc.}

 
The rest of the paper is structured as in the following. Preliminarily, in Section \ref{weak}, we first discuss weak solutions of \eqref{eq1.1} and their compatibility conditions to ensure that the solutions we find in Sections \ref{eq} and \ref{self-similar} are qualified. In Section \ref{S:coord} we introduce the equivariance and the relevant coordinate systems. The detailed analysis on the equivariant standing waves are carried out in Section \ref{eq} (in regular hyperbolic coordinates) and Section \ref{self-similar} (in similarity coordinates). Some statements on the invariant manifold theorem used in the paper are given in the Appendix.

\section{Weak solution and compatibility condition.}\label{weak}
In this section, we define weak solution of system \eqref{eq1.1} and give sufficient and necessary compatibility conditions to be satisfied by a class of solutions which include the standing waves and self-similar standing waves found in later sections.\\
\ \\
\noindent {\bf Weak Solution.} We say $(u,\phi)$ is a weak solution of \eqref{eq1.1} on $I\times \R^2$, where $I \subset \R$ is an open interval,  if
\begin{equation}\label{eq2.1}
u\in \mathbb S^2\ \mbox{or}\ \mathbb H^2\ \mbox{a.e.},  \ u, \ \nabla u, \ \phi, \ u\wedge_\pm \nabla u, \ \phi \nabla u, \ \langle u, u_x \wedge_\pm u_y \rangle_\pm  \in L_{loc}^1 (I \times \mathbb R^2),
\end{equation}
and
\begin{equation}\label{eq2.2}
\left\{\begin{aligned}
&\int_{I \times \R^2} \langle u\wedge_\pm u_y, \psi_y\rangle_\pm -\langle u\wedge_\pm u_x, \psi_x\rangle_\pm + \langle u, \psi_t\rangle_\pm \\
&\hspace{3cm}-\phi \big(\langle u_y, \psi_x \rangle_\pm -\langle u_x, \psi_y\rangle_\pm \big)\ dtdxdy=0,\\
&\int_{I \times \R^2}(\rho_{xx}-\rho_{yy})\phi+2\rho \langle u, u_x\wedge_\pm u_y\rangle_\pm \ dtdxdy=0,
\end{aligned}\right.
\end{equation}
for any $\psi\in C_c^\infty(I \times \mathbb{R}^{2},\R^3)$ and $\rho\in C_c^\infty(I \times \mathbb{R}^{2},\R)$.

Before we continue, it is worth pointing out a few useful properties of the Minkowski metric in $\R^3$. Namely,
\begin{equation} \label{Lorenz1}
\langle v, w \wedge_- z\rangle_- = v \cdot \big(\eta_- (w\wedge_- z) \big)    = v \cdot (w \wedge z)
\end{equation}
and for $u, \ v \in \R^3$ satisfying $\langle u, v\rangle_-=0$,
\begin{equation} \label{Lorenz2}
u \wedge_- (u \wedge_- v) =  \langle u, u\rangle_- v.
\end{equation}
These identities are parallel to those in the Euclidean metric. For any $u \in \mathbb S^2$ (or $\mathbb H^2$)  and $v \in \R^3$,
define
\begin{equation} \label{CJ}
\CJ v := \CJ_u v := u \wedge_\pm v.
\end{equation}
The above identities imply
\[
\langle \CJ v, u \rangle_\pm =0 \qquad \langle v, \CJ v\rangle_\pm =0
\]
and
\[
\CJ^2 v = -v  \qquad \langle \CJ v, \CJ v\rangle_\pm = \langle v, v \rangle_\pm \qquad \text{ if } \langle v, u\rangle_\pm =0.
\]
Therefore $\CJ$ provides the almost complex structure on $\mathbb S^2$ or $\mathbb H^2$.

We note that the linear operator $\Box=\partial_x^2-\partial_y^2$ has nontrivial characteristics $\{(k_1,k_2)\big||k_1|=|k_2|\}$, where $k_1,k_2$ are the Fourier variables of $x,y$. Therefore, the solution may have singularities along any surface in the form of
\[
\{(t,x_0(t)+\alpha_1 s,y_0(t)+\alpha_2 s\big|t,s\in \R\}\ , \ \alpha_{1,2}=\pm 1.
\]
A similar situation also occurs in the cubic hyperbolic Schr\"odinger equation in 2D
\[iu_t+\Box u+|u|^2u=0.\]
In \cite{KNZ11}, the authors constructed (hyperbolically) radial standing wave which satisfies certain compatibility conditions along 
\[S=\{(t,x,y)\big||x|=|y|,t\in \R\}.\]
In our case, we choose the same surface and consider solutions which are smooth except along $S$. The cross $S$ separate the plane $\R^2$ into four cones: two horizontal (indexed by `h') and two vertical (indexed by `v'). In the rest of the manuscript, we will use notation like $\phi_\pm^{h, v}$ to denote a function defined on one of the cones where $\pm$ indicates the sign of $x$ in horizontal cones  or the sign of $y$ in vertical cones. 

To consider weak solutions, since $u$ is mapped to a manifold, in general we require $u$ being continuous. In particularly we are interested in weak solutions whose the field function is in the form of

\[
\phi (t,x,y)=\phi_{1\pm}^{h,v}(t,x+y)\log{|x-y|}+\phi_{2\pm}^{h,v} (t,x-y)\log{|x+y|}+\phi_{3\pm}^{h,v}(t,x,y).
\]

In the above form, we assumed that $\phi$ has exactly two orders -- $O(1)$ and the logarithmic order -- in its asymptotic expansion  near the possible singularity $S$. In fact,
\[
\phi_{1+}^h (t, z) = \lim_{x+y \to z\, \&\, x-y \to 0-} \frac {\phi(t, x, y)}{\log|x-y|}, \quad z>0
\]
and similar limiting property holds for all $\phi_{(1,2)\pm}^{h,v}$ as well. The adoption of the logarithmic order is due to the fact that, in the hyperbolic coordinates $(a, \alpha)$ on $\R^2$ to be introduced in Section \ref{S:coord},
\[
\log |a| = \frac 12( \log|x+y| + \log |x-y|), \qquad \alpha = \frac 12 ( \log|x+y| - \log |x-y|)
\]
are solutions to the free wave equation which is the leading linear part of the $\phi$ equation.

Sometimes it is more convenient to work with characteristic coordinate $(\xi, \eta)$
\[\xi=x+y\ , \ \eta=x-y\]

where the field $\phi$ takes the form
\begin{equation}\label{eq2.6}
\phi(t,\xi,\eta)=\phi_{1\pm}^{h,v} (t,\xi)\log{|\eta|}+\phi_{2\pm}^{h,v} (t,\eta)\log{|\xi|}+\phi_{3\pm}^{h,v} (t,\xi,\eta)
\end{equation}
and
\[S=\{(t,\xi,0)\big|(t,\xi)\in I\times \mathbb R\}\cup\{(t,0,\eta)\big|(t,\eta)\in I\times\mathbb R\}.\]

In addition to \eqref {eq2.1}, we also assume that, for each $i \in \{u, l, d, r\}$ and $t \in I$ where $I \subset \R$ is an open interval,
\begin{equation} \label{weakSAssump}
\begin{aligned}
&u \in C^0(I\times \R^2),  \, \text{and away from } S, \,  (u, \phi) \, \text{ is smooth and satisfies \eqref{eq1.1}};\\
&\phi_{(1,2,3)\pm}^{h,v}\ \text{ are locally bounded, and  } \phi_{3\pm}^{h,v} \in C^0 \ \text{ on the closure of the cones};\\
&\phi_{(1,2)\pm}^{h,v} (t, z) = \phi_{(1,2)\pm}^{h,v} (t, 0) + O(|z|^\delta), \ \delta \in (0, 1];\\
&  \exists \ p>1, \, \epsilon_n  \to 0^+ \ \text{ so that on any bounded interval } \ I' \subset \R, \ \text{ it holds } \\
&\quad |u_\xi (t, \cdot, \pm \epsilon_n) - u_\xi (t, \cdot, 0)|_{L_\xi^p(I')} = O(\epsilon_n^\delta)= |u_\eta (t, \pm \epsilon_n, \cdot) - u_\eta (t, 0, \cdot)|_{L_\eta^p(I')}\quad .
\end{aligned}\end{equation}

The $L_{loc}^p (\R)$ convergence of $u_\xi$ as a function of $\xi$ as $\eta \to 0$ (or of $u_\eta$ as a function of $\eta$ as $\xi \to0$) along a sequence $\eta=\pm \epsilon_n$ is a weak assumption which holds automatically if $u$ is piecewise $C^1$ in the four cones and $C^0$ on $\R^2$.

In the rest of the section, we will identify necessary and sufficient conditions which ensure such $(u, \phi)$ satisfy \eqref{eq2.2}.\\

\noindent {\bf Compatibility Conditions.} In characteristic coordinates $(\xi, \eta)$, equation \eqref{eq2.2} can be written as
\begin{equation}\label{eq2.8}
\left\{\begin{aligned}
&\int_{I \times \R^2}-\langle u\wedge_\pm u_\xi, \psi_\eta \rangle_\pm -\langle u\wedge_\pm u_\eta, \psi_\xi \rangle_\pm +\frac 12\langle u, \psi_t\rangle_\pm \\
&\hspace{3cm}-\phi \big(\langle u_\xi, \psi_\eta \rangle_\pm -\langle u_\eta, \psi_\xi\rangle_\pm \big) \ dtd\eta d\xi=0,\\
&\int_{I\times \R^2}\rho_{\xi\eta}\phi+\rho \langle u, u_\eta\wedge_\pm u_\xi\rangle_\pm\ dtd\eta d\xi=0.
\end{aligned}\right.
\end{equation}
In order to state the compatibility conditions, we first introduce some notation. Let
\[
\phi_1^- (t, \xi) = \begin{cases} \phi_{1-}^h (t, \xi) \quad \xi \le 0 \\  \phi_{1+}^v (t, \xi) \quad \xi>0 \end{cases} \qquad \phi_1^+ (t, \xi) = \begin{cases} \phi_{1-}^v (t, \xi) \quad \xi \le 0 \\  \phi_{1+}^h (t, \xi) \quad \xi>0; \end{cases}
\]
\[
\phi_2^- (t, \eta) = \begin{cases} \phi_{2-}^h (t, \eta) \quad \eta \le 0 \\  \phi_{2-}^v (t, \eta) \quad \eta>0 \end{cases} \qquad \phi_2^+ (t, \eta) = \begin{cases} \phi_{2+}^v (t, \eta) \quad \eta \le 0 \\  \phi_{2+}^h (t, \eta) \quad \eta>0 \end{cases}
\]
and
\[
[\phi_{1,2}] = \phi_{1,2}^+ - \phi_{1,2}^-.
\]
Let
\[\begin{split}
&c_1 (t, \xi) \triangleq (\phi_{3+}^v - \phi_{3+}^h)(t, \xi, 0), \quad c_2 (t, \eta)  \triangleq (\phi_{3-}^h - \phi_{3+}^v)(t, 0, \eta),\\
&c_3 (t, \xi)  \triangleq (\phi_{3-}^v - \phi_{3-}^h)(t, \xi, 0), \quad  c_4 (t, \eta) \triangleq (\phi_{3+}^h - \phi_{3-}^v)(t, 0, \eta).
\end{split}\]

\begin{prop} \label{P:compatibility}
Functions $(u, \phi)$ satisfying the above assumptions \eqref{eq2.1} and \eqref{weakSAssump} are weak solutions if and only if, for each $t\in I$,
\begin{align}
&\phi_{1,2}^\pm \; \text{ are continuous at } 0 \; \text{ and } \; [\phi_{1,2}] \; \text{ and } \;  c_{1,2,3,4} \; \text{ depend only on } \; t. \label{eq2.9}\\
& c_2 + c_4=0 \text{ or } c_1 + c_3=0 \text{ at } 0 \text{ or equivalently at } (0,0),  \label{eq2.10}\\ 
& \qquad \qquad \qquad \phi_{3-}^h+\phi_{3+}^h=\phi_{3-}^v + \phi_{3+}^v, \\
& [\phi_1] (t) u_\xi (t, \xi, 0),\; c_{1,3} (t) u_\xi (t, \xi, 0), \; [\phi_2] (t) u_\eta (t, 0, \eta), \; c_{2,4} (t) u_\eta (t, 0, \eta) \text{ vanish} \label{eq2.11}
\end{align}
where the indices in $c_{1,3}$ and $c_{2,4}$ are determined by the sign of $\xi$ and $\eta$. 
\end{prop}

We will prove the above equivalence property in the following manner. First we use the second equation of \eqref{eq2.8}, which is equivalent to \eqref{eq2.12} to derive \eqref{eq2.14}--\eqref{eq2.17} below, which implies \eqref{eq2.9}. Then we show \eqref{eq2.9} implies the field equation (or equivalently the second equation of \eqref{eq2.8}) if and only if \eqref{eq2.10} holds. In summary, the second equation of \eqref{eq2.8} is equivalent to \eqref{eq2.9} and \eqref{eq2.10}. Finally, we derive that the first equation of \eqref{eq2.8} is equivalent to \eqref{eq2.19} which is true if and only if \eqref{eq2.9}--\eqref{eq2.11} hold.
\begin{proof}
We start with the field equation. Since the logarithmic function is locally integrable even near the singularity then, for any test function $\rho$, we have:
\[\begin{aligned}
&\lim_{\epsilon\to0^+}\int_I \int_R\int_{-\epsilon}^{\epsilon}
\big(\rho_{\xi\eta}\phi+\rho \langle u, u_\eta\wedge_\pm u_\xi\rangle_\pm\big)(t,\xi,\eta)\ d\eta d\xi dt=0,\\
&\lim_{\epsilon\to0^+}\int_I\int_R\int_{-\epsilon}^{\epsilon}
\big(\rho_{\xi\eta}\phi+\rho \langle u, u_\eta\wedge_\pm u_\xi\rangle_\pm\big)(t,\xi,\eta)\ d\xi d\eta dt = 0.\end{aligned}\]
Therefore,
\[\lim_{\epsilon\to0^+}\int_I (\int_{-\infty}^{-\epsilon}+\int_{\epsilon}^{\infty})(\int_{-\infty}^{-\epsilon}+\int_{\epsilon}^{\infty})
\big(\rho_{\xi\eta}\phi+\rho \langle u, u_\eta\wedge_\pm u_\xi\rangle_\pm \big)(t,\xi,\eta)\ d\eta d\xi dt= 0.\]
Since $(u,\phi)$ are classical solution away from $S$, integrate the above equation by parts, we have that the second equation of \eqref{eq2.8} holds if and only
\begin{equation}\label{eq2.12}
\begin{aligned}
0=&\lim_{\epsilon\to0^+}\Big(-\int_I (\rho\phi)(t,-\epsilon,-\epsilon)-(\rho\phi)(t,\epsilon,-\epsilon)-(\rho\phi) (t,-\epsilon,\epsilon)+(\rho\phi)(t,\epsilon,\epsilon)\  dt\\
&+\int_I (\int_{-\infty}^{-\epsilon}+\int_\epsilon^\infty)\rho_\xi\phi\big|^{\eta=-\epsilon}_{\eta=\epsilon}\ d\xi dt+\int_{I}(\int_{-\infty}^{-\epsilon}+\int_\epsilon^\infty)\rho_\eta\phi\big|^{\xi=-\epsilon}_{\xi=\epsilon}\ d\eta dt\Big)\\
\triangleq&\lim_{\epsilon\to0^+} -I+II+III.
\end{aligned}
\end{equation}
Choosing $\rho$ to be vanishing on $\eta=0$, namely, $\rho(t,\xi,0)=0$, then
\[|(\rho\phi)(t,\xi,\pm\epsilon)|\leq C|\epsilon\log\epsilon|, \quad |(\rho_\xi\phi)(t,\xi,\pm\epsilon)|\leq C|\epsilon\log\epsilon|, \; \text{ for } \; |\xi|\ge \epsilon
\]
where the constant $C$ 
is independent of $\epsilon$. It implies
\[
|I|+|II|\leq C|\epsilon\log\epsilon|.
\]
Together with \eqref{eq2.12}, we get, for such $\rho$,
\begin{equation}\label{eq2.13}
\begin{aligned}
0=&\lim_{\epsilon\to0^+}\int_I (\int_{-\infty}^{-\epsilon}
+\int_{\epsilon}^{\infty})\rho_\eta(\phi_{1} \log{|\eta|}+\phi_3)\big|^{\xi=-\epsilon}_{\xi=\epsilon}\ d\eta dt\\
&+\lim_{\epsilon\to0^+}\int_I (\int_{-\infty}^{-\epsilon}
+\int_{\epsilon}^{\infty})\rho_\eta\phi_2 \log\epsilon\big|^{\xi=-\epsilon}_{\xi=\epsilon}\ d\eta dt\triangleq \lim_{\epsilon\to0^+}III_1+III_2.
\end{aligned}
\end{equation}
We first note that $|III_1|=O(1)$ and according to \eqref{eq2.6},
\[
\begin{aligned}
III_2=&\int_I (\int_{-\infty}^{-\epsilon}
+\int_{\epsilon}^{\infty})\rho_\eta(t,\xi,\eta)\phi_2(t,\eta)\log\epsilon\big|^{\xi=-\epsilon}_{\xi=\epsilon}\ d\eta dt \\
=&(\log\epsilon)\int_I \ \int_{-\infty}^0 \rho_\eta(t,0,\eta) (\phi_{2-}^h - \phi_{2+}^v)(t,\eta) \ d\eta +\int_0^{\infty} \rho_\eta(t,0,\eta)  (\phi_{2-}^v - \phi_{2+}^h)(t,\eta)   \ d\eta \ dt\\
&+O(\epsilon\log\epsilon).
\end{aligned}\]
Therefore, \eqref{eq2.13} implies
\[
\int_I \int_{-\infty}^0 \rho_\eta(t,0,\eta) (\phi_{2-}^h - \phi_{2+}^v)(t,\eta) \ d\eta +\int_0^{\infty} \rho_\eta(t,0,\eta)  (\phi_{2-}^v - \phi_{2+}^h)(t,\eta)   \ d\eta dt=0
\]
for all $\rho \in C_c^\infty (I\times \R^2)$ such that $\rho(t, \xi, 0)=0$. Clearly, this condition is equivalent to
\begin{equation} \label{eq2.14}
\phi_{2-}^h - \phi_{2+}^v \quad \text{ and } \quad 
\phi_{2-}^v - \phi_{2+}^h \quad \text{ are independent of } \eta.
\end{equation}
Under this condition, we obtain
\[
III_2
=O(\epsilon\log\epsilon),
\]
and thus it has to hold that $|III_1|=o(1)$ as $\epsilon\to0^+$. To

analyze $III_1$, we carry out a similar calculation using the integrability of $\log |\eta|$ and the dominant convergence theorem
\[\begin{aligned}
III_1 = & \int_I \ \int_{-\infty}^0 \rho_\eta(t,0,\eta) \big( (\phi_{1-}^h - \phi_{1+}^v)(t, 0) \log |\eta| + (\phi_{3-}^h - \phi_{3+}^v)(t, 0, \eta)\big) \ d\eta \\
&+\int_0^{\infty} \rho_\eta(t,0,\eta) \big( (\phi_{1-}^v - \phi_{1+}^h)(t, 0) \log |\eta| + (\phi_{3-}^v - \phi_{3+}^h)(t, 0, \eta)\big)  \ d\eta \ dt+o(1 )
\end{aligned}\]
as $\epsilon \to 0$.

Again, \eqref{eq2.13} implies
\[
(\phi_{1-}^h - \phi_{1+}^v)(t, 0) \log |\eta| + (\phi_{3-}^h - \phi_{3+}^v)(t, 0, \eta) \; \&  \; (\phi_{1-}^v - \phi_{1+}^h)(t, 0) \log |\eta| + (\phi_{3-}^v - \phi_{3+}^h)(t, 0, \eta)
\]
are independent of $\eta$ for $\mp \eta >0$. The piecewise continuity assumptions on $\phi_{1,3}$  immediately yield
\begin{equation}\label{eq2.15}
\begin{aligned}
&\phi_{1-}^h (t,0)=\phi_{1+}^v (t,0)\ , \ \phi_{1-}^v(t,0)=\phi_{1+}^h(t,0), \\
& c_2 (t, \eta)  \;  \text{ and }  c_4 (t, \eta)  \; \text{ are independent of } \eta.
\end{aligned}
\end{equation}

Proceeding from \eqref{eq2.12} again and choosing $\rho(t,0,\eta)=0$, we then have that
\begin{equation}\label{eq2.16}
\phi_{1-}^h - \phi_{1-}^v \quad \text{ and } \quad 
\phi_{1+}^v - \phi_{1+}^h \quad \text{ are independent of } \xi.
\end{equation}
and
\begin{equation}\label{eq2.17}
\begin{aligned}
&\phi_{2-}^h (t,0)=\phi_{2-}^v (t,0)\ , \ \phi_{2+}^v (t,0)=\phi_{2+}^h (t,0), \\
& c_1 (t, \xi)  \; \text{ and }  c_3 (t, \xi) \; \text{ are independent of } \xi.
\end{aligned}
\end{equation}
The above conditions  \eqref{eq2.14}--\eqref{eq2.17} immediately imply \eqref{eq2.9}.
This is a necessary condition to ensure \eqref{eq2.12} hold which is equivalent to the second equation of \eqref{eq2.8}, the weak solution criterion for the class of functions $(u, \phi)$ under our consideration.

Under condition \eqref{eq2.9}, we shall analyze \eqref{eq2.12}  for any given test function $\rho$. Using the dominant convergence theorem and \eqref{eq2.9}, one may compute (skipping writing the $t$ variable)
\[\begin{split}
II= & \int_I \int_\R - (\log \epsilon) \rho_\xi (\xi, 0) [\phi_1] (\xi) d\xi +  \int_{-\infty}^0 \rho_\xi (\xi, 0) \Big( \big(\phi_{2-}^h (-\epsilon) - \phi_{2-}^v (\epsilon)\big)  \log |\xi| \\
&\quad + \phi_{3-}^h (\xi, -\epsilon) - \phi_{3-}^v (\xi, \epsilon) \Big) \ d\xi + \int_0^{\infty} \rho_\xi (\xi, 0) \Big( \big(\phi_{2+}^v (-\epsilon) - \phi_{2+}^h (\epsilon)\big)  \log |\xi| \\
&\quad + \phi_{3+}^v (\xi, -\epsilon) - \phi_{3+}^h (\xi, \epsilon)\Big) \ d\xi \ dt + O(\epsilon \log \epsilon) \\
=&  \int_I \int_{-\infty}^0 - \rho_\xi (\xi, 0) c_3 \ d\xi +\int_0^{\infty} \rho_\xi (\xi, 0) c_1 d\xi \ dt + o(1) \\
=&  - \int_I \rho(t, 0, 0) (c_1 + c_3) dt +  o(1).
\end{split}\]
Similarly
\[
III = \int_I \rho(t, 0, 0) (c_2 + c_4) dt +  o(1).
\]
Finally, we compute $I$ using \eqref{eq2.9}, the dominant convergence theorem, and the $O(|z|^\delta)$ decay of $\phi_{1,2}(t, z)$ at $z=0$  (skipping writing the $t$ variable)
\[\begin{split}
I =&O(\epsilon \log \epsilon) + \int_I \rho(0,0) \Big( \big( \phi_1^- (-\epsilon) +\phi_2^- (-\epsilon) -  \phi_1^- (\epsilon) -  \phi_2^+ (-\epsilon) - \phi_1^+ (-\epsilon)  - \phi_2^- (\epsilon) \\
&\quad + \phi_1^+ (\epsilon) + \phi_1^+ (\epsilon) \big) \log \epsilon + \phi_{3-}^h(-\epsilon, -\epsilon) - \phi_{3-}^v(-\epsilon, \epsilon) - \phi_{3+}^v(\epsilon, -\epsilon) + \phi_{3+}^h(-\epsilon, -\epsilon) \Big) dt \\
=&o(1) + \int_I \big( \rho ( \phi_{3-}^h - \phi_{3-}^v - \phi_{3+}^v + \phi_{3+}^h) \big) (0,0) dt.
\end{split}\]
The above three estimates combined imply that, under condition \eqref{eq2.9}, limit  \eqref{eq2.12} holds if and only if
\begin{equation} \label{eq2.18}
( \phi_{3-}^h - \phi_{3-}^v - \phi_{3+}^v + \phi_{3+}^h) (0,0) \equiv 0 \; \text{ or } \; c_1+ c_3 \equiv 0 \; \text{ or } \; c_2 + c_4 \equiv 0
\end{equation}
where the above conditions are all equivalent to \eqref{eq2.10}.

Finally, integrating the first equation of \eqref{eq2.8} by parts along $\xi=\pm \epsilon_n$ and $\eta=\pm \epsilon_n$ where $\epsilon_n$ was given in \eqref{weakSAssump}, we obtain that it is equivalent to
\begin{equation*}
\begin{aligned}
0=&\lim_{n\to \infty}\Big( - \int_{I}(\int_{-\infty}^{-\epsilon_n}+\int_{\epsilon_n}^\infty) \langle  u \wedge_\pm u_\xi +\phi u_\xi, \psi \rangle_\pm \big|_{\eta=\epsilon_n}^{\eta=-\epsilon_n}\ d\xi dt\\
& +\int_{I}(\int_{-\infty}^{-\epsilon_n}+\int_{\epsilon_n}^\infty)\langle  -u \wedge_\pm u_\eta +\phi u_\eta, \psi \rangle_\pm \big|_{\xi=\epsilon_n}^{\xi=-\epsilon_n}\ d\eta dt \Big).
\end{aligned}
\end{equation*}

Since it is assumed in \eqref{weakSAssump} that $u\in C^0$, $u_\xi (t, \cdot, \eta=\pm \epsilon_n) \to u_\xi(t, \cdot, 0)$  in $L_{loc}^p(\R)$, and a similar property for $u_\eta$, the integrals involving $u\wedge_\pm u_\xi$ and $u\wedge_\pm u_\eta$ converge to $0$ as $n\to \infty$ and thus \eqref{eq2.8} holds if and only if
\begin{equation}\label{eq2.19}
\begin{aligned}
&\lim_{n \to \infty}\Big(\int_{I}(\int_{-\infty}^{-\epsilon_n}+\int_{\epsilon_n}^\infty) \langle (\phi u_\xi), \psi \rangle_\pm \big|_{\eta=\epsilon_n}^{\eta=-\epsilon_n}\ d\xi dt\\
&\hspace{3cm}-\int_{I}(\int_{-\infty}^{-\epsilon_n}+\int_{\epsilon_n}^\infty) \langle \phi u_\eta, \psi\rangle_\pm \big|_{\xi=\epsilon_n}^{\xi=-\epsilon_n}\ d\eta dt\Big)=0.
\end{aligned}
\end{equation}
We will focus on the first integral and the second one can be treated similarly. By \eqref{weakSAssump} (in particular the $L^p$ convergence of $u_\xi$ as $\eta \to 0$) and \eqref{eq2.17}, we have
\[\lim_{n \to \infty}\int_{I}(\int_{-\infty}^{-\epsilon_n}+\int_{\epsilon_n}^\infty) \langle \phi_2(t,\eta)\log{|\xi|} u_\xi, \psi\rangle_\pm \big|_{\eta=\epsilon_n}^{\eta=-\epsilon_n}\ d\xi dt=0.
\]
Again from \eqref{weakSAssump} and \eqref{eq2.9}, one can compute  (skipping writing the $t$ variable)
\[
\int_{I}(\int_{-\infty}^{-\epsilon_n}+\int_{\epsilon_n}^\infty)\phi_1(\xi) \langle u_\xi, \psi\rangle_\pm \big|_{\eta=\epsilon_n}^{\eta=-\epsilon_n}\ d\xi dt= - \int_{I} \int_\R [\phi_1] \langle u_\xi, \psi\rangle_\pm |_{\eta=0} d\xi + O(\epsilon_n^\delta).
\]
Finally, \eqref{eq2.9} and \eqref{weakSAssump} imply
\[
\int_{I}(\int_{-\infty}^{-\epsilon_n}+\int_{\epsilon_n}^\infty)\phi_3 \langle u_\xi, \psi \rangle_\pm \big|_{\eta=\epsilon_n}^{\eta=-\epsilon_n}\ d\xi dt=  \int_{I} \big(- \int_{\R^-}  c_3 + \int_{\R^+}  c_1\big) 
\langle u_\xi, \psi \rangle_\pm \big|_{\eta=0} d\xi dt + o(1).
\]
Combining the above equalities we obtain
\[\begin{split}
\int_{I}(\int_{-\infty}^{-\epsilon_n}+\int_{\epsilon_n}^\infty)\langle \phi u_\xi,& \psi\rangle_\pm \big|_{\eta=\epsilon_n}^{\eta=-\epsilon_n}\ d\xi dt 
=-  |\log \epsilon_n| \int_{I} \int_\R [\phi_1] \langle u_\xi, \psi\rangle_\pm|_{\eta=0} d\xi dt \\
& \qquad \qquad 
+ \int_{I}\big(- \int_{\R^-}  c_3 + \int_{\R^+}  c_1\big)   \langle u_\xi, \psi\rangle_\pm \big|_{\eta=0} d\xi dt + o(1).
\end{split} \]
Much similarly, the other integral in \eqref{eq2.19} can also be computed
\[\begin{split}
\int_{I}(\int_{-\infty}^{-\epsilon_n}+\int_{\epsilon_n}^\infty)\langle \phi u_\eta, \psi\rangle_\pm \big|_{\xi=\epsilon_n}^{\xi=-\epsilon_n}& d\eta dt 
=- |\log \epsilon_n| \int_{I} \int_\R [\phi_2] \langle u_\eta, \psi\rangle_\pm |_{\xi=0} d\eta dt\\
& + \int_{I} \big( \int_{\R^-}  c_2 - \int_{\R^+}  c_4\big) \langle u_\eta, \psi\rangle_\pm \big|_{\xi=0} d\eta dt + o(1).
\end{split}\]
By considering all test functions, it is easy to see that, in addition to \eqref{eq2.9} and \eqref{eq2.10}, \eqref{eq2.19} implies \eqref{eq2.11} .
\end{proof}

\section{Equivariant maps and coordinates} \label{S:coord}

Suppose groups $G_M$ and $G_N$ act on two sets $M$ and $N$, respectively, and $\rho:
G_M\to G_N$ is a homomorphism. A mapping $m:M\to N$
is called $\rho$-equivariant if
\begin{equation}\label{eq3.1}
m\circ f=\rho(f)\circ m\ , \  f\in G_M.
\end{equation}

For the Ishimori system \eqref{eq1.1}, let $M=\mathbb R^{1+2}$, $N=\mathbb S^2(\text{or } \mathbb H^2)\times \mathbb R$, then solutions are maps from $M$ to $N$. We note that this system is invariant under the following commutative group actions on $M$ and $N$, respectively,
\[\begin{aligned}
& G_M=\Bigg\{f_{t_0,\alpha_0}(t,x,y)=\big(t+t_0,\begin{pmatrix}
\cosh{\al_0} & \sinh{\al_0}\\
\sinh{\al_0} & \cosh{\al_0}
\end{pmatrix}\begin{pmatrix}
x\\y
\end{pmatrix}\big)\Big|(t_0,\alpha_0)\in\mathbb R^2\Bigg\},\\
& G_N=\Bigg\{\tilde f_{\tau_0,\alpha_0}(u_0,u_1,u_2,\phi)=\big(u_0,\begin{pmatrix}
\cos{\tau_0} & -\sin{\tau_0}\\
\sin{\tau_0} & \cos{\tau_0}
\end{pmatrix}\begin{pmatrix}
u_1\\u_2
\end{pmatrix},\phi+\alpha_0\big)\Big|(\tau_0,\alpha_0)\in\mathbb R^2\Bigg\}.
\end{aligned}\]
The invariance of  \eqref{eq1.1} under the action of $f_{t_0, 0}$ on $M$ (invariance under translation in time) and under the action of $\tilde f_{\tau_0, \alpha_0}$ on $N$, due to the facts that $\tilde f_{\tau_0, \alpha_0}$ acts on $\mathbb S^2$ (or $\mathbb H^2$) isometrically and that only $\nabla \phi$ appears in the equations,  is obvious. Though one may verify directly, it is easier to see the invariance of  \eqref{eq1.1} under the action of $f_{0, \alpha_0}$ in hyperbolic coordinates $(a, \alpha)$. \\

\noindent{{\bf Hyperbolic coordinates on $\R^2$.}} For $|x| \ne |y|$, we let
\begin{equation}\label{eq3.2}
\begin{aligned}&x=a\cosh{\alpha}\ , \ y=a\sinh{\alpha}\ , \ \mbox{when}\
|x|>|y|,   \\
&x=a\sinh{\alpha}\ , \ y=a\cosh{\alpha}\ , \ \mbox{when}\ |x|<|y|.
\end{aligned}
\end{equation}
Obviously, the transformation is not well defined along the cross
$|x|=|y|$, which actually correspond to $(|a|=0, |\alpha| = \infty)$. One may easily compute, for $\pm (|x|-|y|)>0$,
\[\begin{split}
&\nabla a=\pm (\frac x a,-\frac y a),\  \ \nabla\alpha=\pm (-\frac{y}{a^2},\frac{x}{a^2}), \ \ a_x^2-a_y^2=\pm 1,\  \  \alpha_x^2-\alpha_y^2=\mp\frac{1}{a^2} \\
& D_x\p_x-D_y\p_y = \pm (D_a\p_a+\frac 1a\p_a-\frac{1}{a^2}D_{\alpha}\p_\alpha), \ \ \Box a=\pm \frac 1 a,\  \ \Box\alpha=0.
\end{split}\]

In terms of $(a,\al)$, \eqref{eq1.1} takes the following form in the horizontal/vertical cones (i.e. $|x|> |y|$ or $|x|<|y|$) :
\begin{equation}\label{eq3.3}
C^h:\ \left\{\begin{aligned} &u_t=\mathcal{J}(D_au_a+\frac 1a
u_a-\frac{1}{a^2}D_\alpha u_\alpha) +\frac{1}{a}(u_\alpha\phi_a - u_a\phi_\alpha),\\
&\phi_{aa}+\frac 1a\phi_a-\frac{1}{a^2}\phi_{\alpha\alpha}=-\frac 2a
\langle u, u_a\wedge_\pm u_\alpha\rangle_\pm,
\end{aligned}\right.
\end{equation}

\begin{equation} \label{eq3.5}
C^v:\ \left\{\begin{aligned}
&u_t=-\mathcal{J}(D_au_a+\frac 1a u_a-\frac{1}{a^2}D_\alpha
u_\alpha)
-\frac{1}{a}(u_\alpha\phi_a - u_a\phi_\alpha),\\
&\phi_{aa}+\frac 1a\phi_a-\frac{1}{a^2}\phi_{\alpha\alpha}=-\frac 2a
\langle u, u_a\wedge_\pm u_\alpha\rangle_\pm,
\end{aligned}\right.
\end{equation}
where $\mathcal J$ is the complex structure on $\mathbb{S}^2$ or
$\mathbb{H}^2$ defined in \eqref{CJ} which is simply counterclockwise rotation by $\frac \pi2$ on the tangent space.

It is clear that the above equations, which are equivalent to \eqref{eq1.1}, are invariant under the translation in $\alpha$ which is equivalent to the action $f_{0, \alpha_0}$.

For any $(\mu,k,c)\in\mathbb R^3$, define a homomorphism $\rho_{\mu, k, c}: G_M \to G_N$ as
\[
\rho (f_{t_0,\alpha_0})=\tilde f_{\mu t_0+k\alpha_0,c\alpha_0}.
\]
\smallskip
We say a map $(u, \phi)$ is $(\mu, k, c)$-equivariant or simply equivariant if
\begin{equation} \label{equiva}
(u, \phi) \circ f_{t_0, \alpha_0} = \rho(f_{t_0, \alpha_0}) (u, \phi), \quad \forall t_0, \ \alpha_0 \in \R.
\end{equation}
When $k=0$, an equivariant solution becomes radial in hyperbolic coordinates.
\medskip

\begin{lemma} \label{L:equiva}
Suppose $(u, \phi)$ are $(\mu, k, c)$-equivariant satisfying that $u\in C^0(\R^{1+2})$ and $u_0(t, 0, 0)>0$. It holds
\begin{enumerate}
\item If $k=0$, then $u(t, x,\pm x) =u(t, 0, 0)$ for all $t, \ x$.
\item If $k\ne 0$, then $u_0(t, x, \pm x)=1$ and $u_1 (t, x, \pm x) = u_2(t, x, \pm x) =0$ for all $t, \ x$.
\item If in addition $u$ is piecewise $C^1$ with possible derivative discontinuity along $\{|x| = |y|\}$, then in each cone separated by $\{|x|=|y|\}$, $\p_x u(t, 0, 0) = \p_y u(t, 0, 0) =0$ for all $t$.
\end{enumerate}
\end{lemma}

\begin{proof}
Obviously statement (3) is only a corollary of (1) and (2). Taking $t_0=0$ in \eqref{equiva} and letting $\lambda_\pm = \cosh \alpha_0 \pm \sinh \alpha_0 =e^{\pm \alpha_0}$, we obtain
\[
u(t, \lambda_\pm x, \pm \lambda_\pm x) =  \Big(u_0, \begin{pmatrix} \cos k\alpha_0 & -\sin k \alpha_0 \\ \sin k\alpha_0 & \cos k \alpha_0 \end{pmatrix} \begin{pmatrix} u_1\\ u_2 \end{pmatrix} \Big) (t, x, \pm x).
\]
If $k=0$, it is immediately clear that $u(t, x, \pm x) =u(t, 0,0)$ and thus conclusion (1) follows.

If $k\ne 0$, taking $x=0$ in the above equality implies $u(t,0,0)$ is invariant under the rotation on $\mathbb S^2$ (or $\mathbb H^2$) and thus it has to be a rotation center which implies $u_{1,2} (t, 0, 0) =0$. Moreover, that equality also implies that the image of $\{(t, x, \pm x) \mid x \in \mathbb R\}$ is a pair of curves both containing the rotation center and invariant under the rotation. Therefore statement (2) follows.
\end{proof}

Since rotations are involved in the equivariance, it is natural to introduce the `polar' coordinates on the target surfaces. \\

\noindent {\bf `Polar' coordinates on surfaces of revolution.}
Let $N$ be a connected complete surface, $O\in N$, $exp_O$ be the exponential map from $T_O N$ to $N$. Let $(s, \beta)$ be a polar coordinate on $T_O N$ centered at the origin. We say $N$ is a surface of revolution, i.e. rotationally invariant, if the rotation by $\beta_0$ on $N$
\begin{equation} \label{R}
R(\beta_0) = exp_O (s, \beta +\beta_0), \; \text{ where } \; exp_O (s, \beta) = p
\end{equation}
is an isometry for all $\beta_0$. For any such surface, $(s, \beta)$ through $exp_O$, for $\beta \in S^1$ and $s \in (0, s_0)$ where $s_0 \in (0, +\infty]$ being the radial diameter, is a global coordinate system on $N$ except for $O$ and at most one more point -- the point realizing the maximal distance $s_0$ to $O$ on $N$ if $s_0 <\infty$. The Riemannian structure on $N$ is given in terms of $(s, \beta)$
\[
ds^2 + \Gamma(s)^2 d\beta^2
\]
where $\Gamma(s)$ satisfies
\[
\Gamma \; \text{ is odd, } \; \ 2s_0\text{-periodic, }  \; \Gamma(s)=0 \; \text{ iff } \; s=n s_0, \;\ n \in \mathbb Z, \;\Gamma_s(n s_0) =(-1)^n, \; \Gamma_{ss} (n s_0) =0.
\]
When $s_0<\infty$, it is only for convenience to extend the definition of $\Gamma$ to an odd and periodic function. The condition $\Gamma_{ss} (ns_0)=0$ is to ensure continuity of the curvature. When $s_0=\infty$, the above property becomes simply $\Gamma$ being odd, $\Gamma_s (0)=1$ and $\Gamma(s)=0$ if and only if $s=0$. Special examples of this type of surfaces are $\mathbb S^2$ and $\mathbb H^2$ where $O =(u_0=1, u_1=0, u_2=0)$ and
\begin{equation} \label{Gamma}
\Gamma_{\mathbb S^2} (s) = \sin s, \quad s \in [0, s_0=\pi], \qquad \Gamma_{\mathbb H^2} (s) = \sinh s, \quad s \in [0, \infty).
\end{equation}
It is straight forward to compute
\begin{equation}\label{eq3.10}
\mathcal{J}\p_s=\frac{1}{\Gamma}\p_\beta\ , \
\mathcal{J}\p_\beta=-\Gamma \p_s
\end{equation}
and the connection
\begin{equation}\label{eq3.11}
D_{\p_s}\p_s=0,\ , \ D_{\p_s}\p_\beta=D_{\p_\beta}\p_s=\frac {\Gamma_s}{\Gamma}  \p_\beta,\ \ D_{\p_\beta}\p_\beta=-\Gamma \Gamma_s \p_s.
\end{equation}
Moreover, for $N = \mathbb S^2$ or $\mathbb H^2$, $u \in N$, $X, Y \in T_u N$,
\[
\langle u, X \wedge_\pm Y\rangle_\pm = \mathcal{J} X \cdot Y.
\]

\section{Equivariant Standing waves.}\label{eq}

In this section, we study equivariant standing wave solutions of \eqref{eq1.1} which are weak solutions. The equivariance is in the sense of \eqref{equiva}. In terms of the hyperbolic coordinates $(a, \alpha)$ given in \eqref{eq3.2} in each of the four cones in $\R^2$ and the polar coordinates $(s, \beta)$ on the target manifold (with the metric given by $\Gamma$ and $s_0$ the radial diameter) introduced in Section \ref{S:coord}, such solutions are in the form of
\begin{equation} \label{eq4.1.1}
\big(u, \phi\big) (t, a, \alpha) = \big(R(\mu t+k\alpha)W(a),\psi(a)+c\alpha\big) \; \text{ and } \; W(a) = exp_O \big(s(a), \beta(a) \big),
\end{equation}
where $R(\cdot)$ is the rotation defined in \eqref{R}. Moreover, Lemma \ref{L:equiva} implies that when $u$ is $C^0$, it satifies
\begin{equation} \label{eq4.9}
W(0) =(1, 0, 0) \; \text{ or equivalently }\; s(0)=0 \; \text{ if } \; k\ne 0.
\end{equation}
Moreover, if $u$ is piecewise $C^1$ on $\{|x|\ne |y|\}$, then
$W_a(0)=0$.
As in Section \ref{weak}, we use notations like $u_\pm^{h,v}$ {\it etc.} to denote quantities $u$ restricted in the horizontal and vertical cones with the sign $\pm$ determined by the sign of $x$ and $y$, respectively. Our main results are as follows:

\begin{mainthm}\label{M1}
For any given $(\mu,k,c)\in\R^3$ none of which vanishes,
system \eqref{eq1.1} admits nontrivial $(\mu,k,c)$-equivariant standing wave solution solutions from $\R^{1+2}$ to $\mathbb S^2$ or $\mathbb H^2$ which are weak solutions. 
\[\begin{aligned}
u_\pm^{h,v} (t,x,y)& =\big(u_{0\pm}^{h,v}, u_{1\pm}^{h,v}, u_{2\pm}^{h,v}\big)\\
&=\big(\Gamma_s(s_\pm^{h, v} (|a|)), \Gamma(s_\pm^{h,v} (|a|)) \cos{(\beta_\pm^{h, v} (|a|)+ k\alpha+\mu t)}, \\
&\hspace{4.2cm}\Gamma(s_\pm^{h, v}(|a|))\sin{(\beta_\pm^{h,v} (|a|)+ k \alpha+\mu t)}\big),\\
\phi_\pm^{h,v}(t,x,y)&=c\alpha + b \log |a| - \int_0^{|a|}\frac{2k}{a'}F(s(a'))\ da',
\end{aligned}\]
where 
\[
|a| = \sqrt{|x^2 - y^2|}, \quad \alpha = \frac 12 \log \frac {|x+y|}{|x-y|}, 
\]
$b$ is any constant satisfying $-4(k^2 + bk) > c^2$ and $F(s) = \int_0^s \Gamma(s') ds'$.  
Moreover, 
\[
s_\pm^{h, v} (0)=0, \qquad s_-^{h, v} (a) = s_+^{h, v} (a), \quad \beta_-^{h, v} (a) = \beta_+^{h, v} (a), \quad \forall a\ge 0
\]  
and $s_\pm^{h, v} (a)$ and $\beta_\pm^{h,v} (a)$ are smooth for $a>0$ and are locally $C^\kappa$ 
for $a\ge 0$ with the 
exponent $\kappa = \sqrt{-(k^2+bk)-\frac{c^2}{4}}$.
\end{mainthm}

We further obtain asymptotic expansion of $s_\pm^{h, v} (a)$ and $\beta_\pm^{h,v} (a)$ of these solutions in terms of the (hyperbolic) radial variable $a$ near  $\infty$. 

\begin{mainthm}\label{M2}
Given a triple $(\mu,k,c)$, none of which vanishes, and a constant $b$ such that $0<c^2<-4(k^2+b k)$. The solutions given in Theorem \ref{M1} satisfy for $a>>1$,
\begin{enumerate}
\item[(1)] 
If $\mu<0$, there exist $E_\infty>0$ such that,
\[\begin{aligned}
&s_\pm^h (a)=\sqrt{\frac{2 E_\infty}{a|\mu|}}\cos{\big(\theta_0-|\mu|^{\frac 12}a-\frac{|\mu|^{-\frac 12}\Gamma'''(0) E_\infty}{8}\log {|a|}+O(a^{-1})\big)}+O(a^{-\frac 32}),\\
&\p_a s_\pm^h (a)=\sqrt{\frac{2 E_\infty}{a}}
\sin{(\theta_0-|\mu|^{\frac 12}a-\frac{|\mu|^{-\frac 12}\Gamma'''(0)E_\infty}{8}\log {|a|}+O(a^{-1}))}+O(a^{-\frac 32}),\\
&(\Gamma\p_a \beta_\pm^h)(a)=-c \sqrt{\frac{E_\infty}{2a^3|\mu|}}
\cos{\big(\theta_0-|\mu|^{\frac 12}a-\frac{|\mu|^{-\frac 12}\Gamma'''(0)E_\infty}{8}\log{|a|}+O(a^{-1})\big)}+O(a^{-\frac 52}).
\end{aligned}\]
If $\mu>0$, then $s_\pm^v$ and $\beta_\pm^v$ satisfy the above estimates. \\
\item[(2)] Suppose the target manifold is $\S^2$ (compact) and $\mu>0$, then 
we have either
\[\begin{aligned}
&s_\pm^h (a)= \pi - \sqrt{2} \mu^{-\frac 14} |c|^{\frac 12} a^{-\frac 12} 
+O(a^{-\frac 32}), \quad s_a(a)=O(a^{-1}),\\
&(\Gamma\p_a \beta_\pm^h)(a)= - \frac {\sqrt{2}c \mu^{\frac 14} }{\sqrt{a|c|}} + O(a^{-\frac 32}).
\end{aligned}\]
or there exist constants $\bar E_\infty > 2\sqrt{\mu} |c|$ and $\theta_0$ such that 
\[\begin{aligned}
&s_\pm^h(a)= \pi - \frac 1{\sqrt {\mu a}} \sqrt {\bar E_\infty+ (\bar E_\infty^2-4\mu c^2)^{\frac 12}  \sin (\theta_0 + 2 \mu^{\frac 12} a) }+O(a^{-1}), \\
&\p_a s_\pm^h (a)= -a^{-\frac 12} \sqrt {\frac {\bar E_\infty^2-4\mu c^2}{\bar E_\infty+ (\bar E_\infty^2-4\mu c^2)^{\frac 12}  \sin (\theta_0 + 2 \mu^{\frac 12} a)}} \cos (\theta_0 + 2 \mu^{\frac 12} a)+O(a^{-\frac 32}),\\
&(\Gamma\p_a\beta_\pm^h)(a) = - \sqrt {\mu} c F(\pi)  a^{-\frac 12}  \Big(\bar E_\infty+ (\bar E_\infty^2-4\mu c^2)^{\frac 12}  \sin (\theta_0 + 2 \mu^{\frac 12} a)\Big)^{-\frac 12} 
+O(a^{-\frac 32}).
\end{aligned}\]
Suppose the target manifold is $\S^2$ (compact) and $\mu<0$, then $s_\pm^v$ and $\beta_\pm^v$ satisfy the above estimates. \\
\item[(3)] Suppose the target manifold is $\H^2$ (noncompact). For the solutions obtained in Theorem \ref{M1}, $s_\pm^h \equiv 0$ if $\mu>0$ or $s_\pm^v \equiv 0$ if $\mu<0$. 
\end{enumerate}
\end{mainthm}

In particular, the above asymptotic expansions, which actually hold for general target manifolds with rotational symmetry, indicate that these solutions are nontrivial. Geometrically, these solutions map each branch of hyperbola $x^2-y^2=h $ to a circle centered at the rotation center on the target manifold. Each ray in $\R^2$ other than $\{|x| = |y|\}$ is mapped to a curve on the target manifold starting at the rotation center. From one pair of cones (the horizontal pair or the vertical pair),  these curves end at the same rotation center (the above case (1)). From the other pair of cones, these curves either are taken to be a single point -- the rotational center -- in the case of non-compact target manifolds (the above case (3)) or end at the other rotation center (the above case (2)) in the case of compact target manifolds. Therefore, the image of the quadrants in case (1) can be viewed as a retractable bubble, while those of case (2) are not retractable and occupy the whole target manifold. Moreover, in the case (1), the image curves of the rays converge to the limit along a single direction as $a \to \infty$ and, in the case (2), they converge in a spiraling fashion.  On the other hand,  in the case when the target manifolds are non-compact like $\H^2$ and $\mu>0$, we can not rule out equivariant solutions which are non-trivial in horizontal cones as the corresponding ODE becomes too technical to be analyzed and its solutions may not always exist globally in $a$. We took trivial solutions in those cases for simplicity. 

In Subsection \ref{SS:pre}, we rewrite the problem into ODE systems and derive some of their basic properties. Main Theorem \ref{M1} is proved in subsection \ref{SS:local} and \ref{SS:global1}. Asymptotic 
estimates in Main Theorem 2 are obtained in Subsection \ref{SS:negative-mu}  (case (1)) and 
Subsection \ref{SS:positive-mu} (case (2)). 
In Subsection  \ref{SS:VWeak} we verify that the solutions obtained in this section are weak solution of the generalized Ishimori system. Except for this subsection, most of the arguments in this section are carried out for general surfaces of revolution with the metric given by a general $\Gamma$ as described in the last section. 

Finally, we discuss radial solutions (in hyperbolic coordinates) in Subsection \ref{SS:radial} and prove that there do not exist equivariant standing waves which are weak solutions of \eqref{eq1.1} mapping the cross $\{|x|=|y|\}$ to the rotation center.

\subsection{Preliminaries} \label{SS:pre} 
In this ansatz \eqref{eq4.1.1}, the Ishimori system, which can be rewritten as equations \eqref{eq3.3} and \eqref{eq3.5}, is equivalent (in hyperbolic coordinate system) to 
\begin{equation}\label{eq4.2}
C^h:\left\{\begin{aligned} &\mu\p_\beta=\mathcal{J}(D_aW_a+\frac
1a W_a-\frac{k^2}{a^2}D_{\p_\beta} \p_\beta)
+\frac{1}{a}(k\psi_a\p_\beta - cW_a),\\
&\psi_{aa}+\frac 1a\psi_a=-\frac {2k\Gamma(s) }{a} s_a,
\end{aligned} 
\right.
\end{equation}

\begin{equation} \label{eq4.4}
C^v:\left\{\begin{aligned}
&\mu\p_\beta=-\mathcal{J}(D_aW_a+\frac 1a
W_a-\frac{k^2}{a^2}D_{\p_\beta} \p_\beta)
-\frac{1}{a}(k\psi_a\p_\beta - cW_a),\\
&\psi_{aa}+\frac 1a\psi_a=-\frac {2k\Gamma(s)}{a} s_a,
\end{aligned}
\right.
\end{equation}
where $\Gamma(s) = \sin s$ for $\mathbb S^2$ and $\Gamma(s) = \sinh s$ for $\mathbb H^2$ and $s(a) = s\big(W(a)\big)$ is the $s$ coordinate of $W(a)$. However, in most of the arguments, we do not have to limit to these two manifolds. 

\begin{rem}\label{rem5.1}
Observations: 
\begin{enumerate} 
\item If $\big(W(a), \psi(a)\big)$ is a solution of \eqref{eq4.2} with parameters $(\mu, k,b, c)$, then $\big(W(a), \psi(a) \big)$ is also a solution of \eqref{eq4.4} with parameters $(-\mu, k, b,c)$. 
\item \eqref{eq4.2} and \eqref{eq4.4} are both reversible, i.e. if $\big(W(a), \psi(a)\big)$, $a>0$ is a solution of \eqref{eq4.2} with parameters $(\mu, k,b, c)$, then $\big(W(-a), \psi(-a) \big)$, $a<0$, is also a solution of \eqref{eq4.2}. 
\end{enumerate}
Therefore, we will focus on \eqref{eq4.2} for $a>0$.
\end{rem}

Before starting to analyze \eqref{eq4.2}, we first make two observations. Firstly, the form of $\phi_\pm^{h,v}$ of solutions presented in Main Theorem \ref{M1}  implies that $\psi_\pm^{h,v}$ in \eqref{eq4.1.1} satisfies 
\begin{equation} \label{H1}
b = \displaystyle\lim_{a\to0} a\p_a\psi_\pm^{h, v} (a) \text{ exists and are identical for the four cones.} 
\end{equation}
In the rest of the is section, we will work under this assumption. This assumption/properties, which admits logarithmic singularity of $\phi$, is in line with the form \eqref{eq2.6} of weak solutions we considered in Section \ref{weak}. In fact, one may start working with $b_\pm^{h,v} = \lim_{a\to0\pm} a\p_a\psi_\pm^{h, v} (a)$, but it turns out that $b_\pm^{h,v}$ arising from different cones must be the same in order for the compatibility condition \ref{P:compatibility} to be satisfied. 

Secondly, recall our goal is to find equivariant standing waves qualified to be weak solutions of \eqref{eq1.1}. They satisfy \eqref{eq4.2} and \eqref{eq4.4} only in the interior of horizontal and vertical cones (i.e. $a \ne 0$), respectively. In order to verify what we will find in the following subsections are actual weak solutions, one necessary condition is $\nabla u \in L_{loc}^1$ as required in \eqref {eq2.1}. In terms of the characteristic coordinates $(\xi, \eta)$ which satisfy in the right side horizontal cone $\{\xi, \eta >0\}$
\[
\xi = x+y, \quad \eta = x-y, \quad a = \sqrt{\xi \eta}, \quad \alpha = \frac 12 (\log \xi - \log \eta),
\]
one may compute for equivariant maps given in \eqref{eq4.1.1}
\[
|u_\xi| = |a_\xi W_a + k \alpha_\xi \p_\beta| = \frac 1{2\xi} |a W_a + k \Gamma \CJ \p_s|.    
\]
Therefore it is natural (at least for the case $k\ne 0$) to look for solutions satisfying assumptions 
\begin{equation} \label{AssumAtZero} 
\lim_{a\to 0\pm} a W_a (a) =0 \quad \text{ and } \quad s(0)=0  
\end{equation} 
in the rest of this section.  

From assumption \eqref{H1} and the second equation of \eqref{eq4.2}, it is easy to see
\begin{equation} \label{si}
\psi_a (a) = \frac {b}a - \frac {2k}a F\big(s(a)\big) \; \text{ where } \; F(s) = \int_0^s \Gamma(s') ds'.
\end{equation}
Substituting it into the first equation of \eqref{eq4.2} and using $\p_\beta = \Gamma \mathcal {J} \p_s$, we obtain
\begin{equation}\label{eq4.7}
D_aW_a+\frac 1a W_a +\big( \frac{k^2}{a^2}\Gamma \Gamma' - \frac {2k^2}{a^2} F\Gamma + \frac {b k}{a^2} \Gamma - \mu \Gamma\big) \p_s  + \frac ca \mathcal{J} W_a=0.
\end{equation}
It seems that this equation is well defined only for $s\ne 0$ (and $s \ne \pi$ as well for $\mathbb S^2$) where $\p_\beta$ is defined. However, this is only a coordinate singularity which can be removed by rewriting the equation. Let, for $W\in \mathbb S^2$ or $\mathbb H^2$,
\[
G (W) = G(s)=  \frac {k^2}2 \Gamma^2 + {b} kF - k^2 F^2, \qquad F(W) = F(s), \qquad s= s(W). 
\]
Even though the radial coordinate $s(W)$ can take multiple values for given $W$, our odd (and periodic in the case of $\mathbb S^2$) extension of $\Gamma$ implies that $G$ and $F$ are smooth functions defined on the surface and thus their gradients define smooth vector fields on the surface. Equation \eqref{eq4.7} can written as
\begin{equation} \label{eq4.8}
D_aW_a+\frac 1a W_a + \frac 1{a^2} \nabla G -\mu \nabla F+  \frac ca \mathcal{J} W_a=0.
\end{equation}
This equation is subject to boundary conditions \eqref{eq4.9} and \eqref{AssumAtZero} at $a=0$. 

\begin{lemma} \label{L1}
Suppose $k^2 + b k > 0$. If a solution $W(a)$ of \eqref{eq4.7} for $a>0$ is continuous at $a=0$ and satisfies $\lim_{a\to 0^+} a |W_a(a)| =0$, then $W(a) \equiv (1,0,0)$.
\end{lemma}

When $k^2+bk=0$, we have either $k=0$ or $b=-k$. The first case corresponds to radial solutions, which will be discussed in subsection \ref{SS:radial}. Our analysis fails in the second case. However, this case is not generic (compared with $b\ne -k$) and has no significant geometric difference as $k=0$. We thus skip this case. In the spirit of this lemma and assumption \eqref{AssumAtZero}, in the rest of this section except subsection \ref{SS:radial}, we will work under the assumption $k^2 + b k < 0$.

\begin{proof}
Since $k\ne 0$, Lemma \ref{L:equiva} implies $W(0)=(1,0,0)$ and thus $s(0)=0$ (Lemma \ref{L:equiva} can be easily modified for general surfaces of revolution). Multiplying \eqref{eq4.8} by $a^2W_a$ and using \eqref{eq3.11}, we obtain
\begin{equation} \label{eq4.11}
E_a =  - 2 \mu a F
\end{equation}
where
\begin{equation} \label{eq4.12}
E(a) = \frac {a^2}2 |W_a|^2 + G -\mu a^2  F.
\end{equation}
For small $|s|<<1$,
\[
F(s) = \frac 12 s^2 + O(s^3) \qquad G(s) = \frac 12 (k^2 + b k) s^2 + O(s^3).
\]
Since $k^2 + b k>0$ the above equality implies that there exist $a_0, \  c_{1,2}>0$ such that for $a\in (0, a_0)$
\[
c_1 s^2 \le E(a) \le 2 |\mu| \int_0^a a' F\big(s(a')\big) da' \le c_2 \int_0^a s(a')^2 da'.
\]
The Gronwall inequality immediately implies that $s\equiv 0$ which competes the proof.
\end{proof}

Since the Ishimori system is invariant under rotations on the target manifold, the angular $\beta$ coordinate of $W$ is not as important as the radial coordinate $s$ and the component of $W_a$ in the radial and angular directions in the tangent space. We first rewrite equation \eqref{eq4.7} or equivalently \eqref{eq4.8}. While $s_a = \langle W_a, \p_s \rangle$, let
\[
\sigma = \langle W_a,  \CJ \p_s\rangle.
\]
One may compute using \eqref{eq3.10}, \eqref{eq3.11},
\begin{equation} \label{eq4.13} \begin{split}
s_{aa} =& \p_a \langle \p_s, W_a\rangle = \langle D_{W_a} \p_s, W_a \rangle + \langle \p_s, D_a W_a \rangle \\
=& \frac \sigma\Gamma \langle D_{\p_\beta} \p_s, W_a \rangle - \langle \p_s, \frac 1a W_a + \frac 1{a^2} \nabla G -\mu \nabla F +\frac ca  \CJ W_a \rangle \\
=& \frac {\Gamma_s}\Gamma \sigma^2 - \frac 1a s_a - \frac 1{a^2} G_s + \mu \Gamma + \frac ca \sigma.
\end{split}\end{equation}
Similarly,
\begin{equation} \label{eq4.14}\begin{split}
\sigma_a = & \p_a \langle \CJ \p_s, W_a\rangle = \langle \CJ D_{W_a} \p_s, W_a \rangle + \langle \CJ \p_s, D_a W_a \rangle \\
=& \frac \sigma\Gamma \langle \CJ D_{\p_\beta} \p_s, W_a \rangle - \langle \CJ \p_s, \frac 1a W_a + \frac 1{a^2} \nabla G -\mu \nabla F +\frac ca  \CJ W_a \rangle \\
=& - \frac {\Gamma_s}\Gamma \sigma s_a - \frac 1a \sigma - \frac ca s_a.
\end{split}\end{equation}
Boundary condition
\eqref{eq4.9} implies
\begin{equation} \label{eq4.15}
s(0)=0 \; \text{ if } \; k\ne 0; \qquad s_a(0)=\sigma (0)=0 \; \text{ if } \; W \; \text{ is } \; C^1 \; \text{ at } \; 0.
\end{equation}
Clearly the above equations form a non-autonomous system (with singularity at $a=0$) of unknowns $(s, s_a, \sigma)$.

Motivated by \eqref{eq4.11} and \eqref{eq4.12}, we rewrite $W_a$ in (scaled) polar coordinates in $T_W \mathbb S^2$ (or $T_W \mathbb H^2$). Let
\[
\frac {a s_a}\Gamma = r \cos \gamma, \qquad  \frac {a \sigma}\Gamma = r \sin \gamma.
\]
One may compute using equations \eqref{eq4.13} and \eqref{eq4.14}
\begin{equation} \label{eq4.16}\left\{\begin{aligned}
r_a = &\frac 1a ( - \Gamma_s r^2 - \frac {G_s}\Gamma + \mu a^2) \cos \gamma\\
r\gamma_a =& \frac 1a ( -\Gamma_s r^2 +  \frac {G_s}\Gamma - \mu a^2) \sin \gamma - \frac ca r.
\end{aligned}\right. \end{equation}
We notice that the term $\frac {G_s}\Gamma$ is smooth at $s=0$ and for $r\ne 0$, the only singularity in the above is $\frac 1a$. To remove this singularity, we change the variable to $\tau = \log a$ and obtain
\begin{equation} \label{eq4.17}\left\{\begin{aligned}
s_\tau =& r \Gamma \cos \gamma\\
r_\tau = &( - \Gamma_s r^2 - \frac {G_s}\Gamma + \mu a^2) \cos \gamma\\
\gamma_\tau  =& ( -\Gamma_s r +  \frac {G_s}{r\Gamma} - \mu \frac {a^2} r ) \sin \gamma - c\\
a_\tau =&a.
\end{aligned}\right. \end{equation}
Depending on the goals in different steps, we will use equations \eqref{eq4.7}, \eqref{eq4.13}, \eqref{eq4.14}, \eqref{eq4.16}, and \eqref{eq4.17} in an intertwined way.

The rest of this section is divided into six subsections. In subsection \ref{SS:local}, we prove the local existence of solutions of \eqref{eq4.8} based on invariant manifold theory. Then we give some global existence results in subsection \ref{SS:global1}. It turns out the sign of $\mu$ and the geometry of the target manifold (compactness and non-compactness) affects the asymptotic analysis. In subsection \ref{SS:negative-mu}, we obtain asymptotic formulae of solutions obtained from subsection \ref{SS:local} at $a=\infty$ for $\mu<0$. The next subsection gives asymptotic formulae of solutions on compact surfaces for $\mu>0$. In subsection \ref{SS:VWeak}, we verify that our equivariant solutions are weak solutions of \eqref{eq1.1} according to \eqref{eq2.1} and \eqref{eq2.2}. Finally, we present some non-existence result for radial solutions in subsection \ref{SS:radial}. 

\subsection{Solution of \eqref{eq4.8} near $a=0$ assuming $b k + k^2<0$} \label{SS:local}

Notice the plane $\Pi \triangleq \{a=0, \, s=0\}$ is invariant under the system \eqref{eq4.17} and a solution of \eqref{eq4.8} satisfying $s(a=0)=0$ corresponds to a solution of \eqref{eq4.17} which converges to $\Pi$ as $\tau \to -\infty$. On $\Pi$, \eqref{eq4.17} is reduced to
\begin{equation} \label{eq4.18} \left\{ \begin{aligned}
r_\tau =& -( r^2 + k^2 + b k) \cos \gamma  \\
\gamma_\tau =& (-r + \frac {k^2 + b k}r) \sin \gamma - c.
\end{aligned} \right. \end{equation}
Recall  that we have assumed, 
\[
k^2 + b k \le 0
\]
\smallskip
 
\noindent due to Lemma \ref{L1},  and thus $\{ r^2 = - (k^2 + b k)\}$ is an invariant circle of \eqref{eq4.18} if $k^2 + b k<0$.
\\

\noindent {\it Case I ($k^2+bk<0,c^2<-4(k^2+bk)$)}. In this case, the invariant circle contains two fixed points
\[(\sqrt{-(k^2+bk)},\gamma_\pm)\ , \ \cos{\gamma_\pm}=\pm\sqrt{1+\frac{c^2}{4(k^2+bk)}}.\]
Linearizing \eqref{eq4.17} at $(s,a,r,\gamma)=(0,0,\sqrt{-(k^2+bk)},\gamma_{\pm})$, the Jacobian matrix is
\begin{equation}\label{eq4.19}
\begin{pmatrix}
\pm\sqrt{-(k^2+bk)-\frac{c^2}{4}} & 0 & 0 & 0\\
0 & \mp2\sqrt{-(k^2+bk)-\frac{c^2}{4}} & 0 & 0\\
0 & 0 & \mp2\sqrt{-(k^2+bk)-\frac{c^2}{4}} & 0\\
0 & 0 & 0 & 1
\end{pmatrix},
\end{equation}
which implies by standard invariant manifold theory that $(0,\sqrt{-(k^2+bk)},\gamma_{\pm},0)$ has $2$ or $3$ dimensional smooth unstable manifold, respectively. In the later case where $\cos{\gamma_-}<0$, the unique unstable manifold is given by $\{s=0\}$. So we look for a nontrivial solution near $(0,\sqrt{-(k^2+bk)},\gamma_+,0)$. The unique local unstable manifold is given by $\{(s,r(s,a),\gamma(s,a),a)\}$ for $(s,a)\in[-s_\star,s_\star]\times [-a_\star,a_\star]$, where
\begin{equation}\label{eq4.20}
r(s,a)=\sqrt{-(k^2+bk)}+O(a^2+s^2)\ , \ \gamma(s,a)=\gamma_++O(a^2+s^2).
\end{equation}
This manifold is locally invariant under the flow of \eqref{eq4.8}, i.e. solutions with initial data on this manifold can only exit this manifold through its boundary.  Moreover, given any solution $(s, r, \gamma, a) (\tau)$ of \eqref{eq4.17}, clearly $(\pm s, r, \gamma, \pm a)(\tau)$ is still a solution and it implies that 
\[
r(s, a) \text{ and } \gamma(s, a) \text{ are even functions}.
\]

It is standard that solutions on the unstable manifold converge to the steady state as $\tau \to -\infty$. We give a more precise description of solutions on this manifold in the following. Let $s=a^\kappa q$, where $\kappa=\sqrt{-(k^2+bk)-\frac{c^2}{4}}$. By the first equation of \eqref{eq4.17}, we obtain
\begin{equation}\label{eq4.21}
q_a=\frac{r(a,a^\kappa q)\Gamma(a^\kappa q)
\cos{\gamma(a,a^\kappa q)}-\sqrt{-(k^2+bk)-\frac{c^2}{4}}a^\kappa q}
{a^{\kappa+1}}\triangleq h(a,q).
\end{equation}
Note that $h\in C^0\big([-a_\star,a_\star]\times [-q_\star,q_\star] \big)$, where $q_\star = a^{-\kappa} s_\star$, is smooth except for $a=0$ and smooth in $q$ everywhere. Moreover, both $h$ and $h_q$ are of order $O(a^{\min\{2\kappa-1,1\}})$. Thus, $h,h_q\in L_a^1L_q^\infty$ for $(a,q)\in[-a_\star,a_\star]\times [-q_\star,q_\star]$. Consequently, for any $a_0 \in (-a_\star, a_\star)$ and $q(a_0)\in(-q_\star,q_\star)$, \eqref{eq4.21} has a unique solution for $a \in (-a_\star, a_\star)$ continuously depending on $q(a_0)$, which is $C^{\min\{2\kappa,1\}}$ in $a$ at $a=0$ 
and $C^1$ everywhere else. Clearly $q(a) \equiv 0$ is the solution for $q(a_0)=0$. 

\begin{rem}
Since we only require $\kappa$ to be positive, the local solution can possibly be only H\"{o}lder continuous at $a=0$.
\end{rem}

\noindent {\it Case II ($k^2+bk<0,c^2=-4(k^2+bk)$)}. In this case, the invariant circle contains only one fixed point. Linearizing \eqref{eq4.17} at such point, the Jacobian matrix is
\[\begin{pmatrix}
0 & 0 & 0 & 0\\
0 & 0 & 0 & 0\\
0 & 0 & 0 & 0\\
0 & 0 & 0 & 1
\end{pmatrix}.\]
In this case, the unique 1-d local unstable manifold coincides with the unstable manifold of the sub-system in the invariant subset $\{s=0\}$ and thus is contained in $\{s=0\}$.\\
\\
\noindent {\it Case III ($k^2+bk<0,c^2>-4(k^2+bk)$)}. In this case, the invariant circle is a periodic orbit. To find solutions asymptotic to this orbit, it is natural to study eigenvalues of the monodromy matrix. If we denote  the fundamental matrix solution of linearized equation around the periodic solution by $D\Phi(t)$, then the monodromy matrix $\mathcal M(t)$ is given by $D\Phi(t+T)(D\Phi(t))^{-1}$, where $T$ is the period of the periodic solution. The Floquet theory shows the eigenvalues of $\mathcal M(t)$ is independent of $t$. To compute the eigenvalues of $\mathcal M(t)$, we reparametrize \eqref{eq4.17} by $\gamma$ in a neighborhood of the circle to obtain
\begin{equation} \label{eq4.22}\left\{\begin{aligned}
s_\gamma =&\frac{ r \Gamma \cos \gamma}{f(s,r,a,\gamma)}\\
r_\gamma = &\frac{( - \Gamma_s r^2 - \frac {G_s}\Gamma + \mu a^2) \cos \gamma}{f(s,r,a,\gamma)}\\
a_\gamma =&\frac{a}{f(s,r,a,\gamma)},
\end{aligned}\right. \end{equation}
where $f(s,r,a,\gamma)=( -\Gamma_s r +  \frac {G_s}{r\Gamma} - \mu \frac {a^2} r ) \sin \gamma - c$. Linearizing such system around the periodic solution
$(s_0,r_0,a_0)=(0,\sqrt{-(k^2+bk)},0)$, one can compute the monodromy matrix is given by
\[\exp\Big({\begin{pmatrix}
\int_0^{2\pi}-\frac{r_0\cos\gamma}{2r_0\sin\gamma+c}\ d\gamma & 0 & 0\\
0 & \int_0^{2\pi}\frac{2r_0\cos\gamma}{2r_0\sin\gamma+c}\ d\gamma & 0\\
0 & 0 & \int_0^{2\pi}-\frac{1}{2r_0\sin\gamma+c}\ d\gamma
\end{pmatrix}}\Big)=\begin{pmatrix}
1 & 0 & 0\\
0 & 1 & 0\\
0 & 0 & \lambda
\end{pmatrix},\]
where $\lambda>1$ if $c<0$. It follows the periodic orbit can at most have one unstable direction, which is in $a$-direction. Since $\{s=0\}$ is again invariant, there is no nontrivial solution.
\begin{rem}We obtain local solutions of \eqref{eq4.17} by constructing the unstable manifold of some fixed point. In the degenerate case of $b=-k$, one needs a different approach rather than change of coordinates leading to \eqref{eq4.17}.
\end{rem}

We summarize the local existence result in the following lemma.

\begin{prop} \label{P:local}
Suppose $c^2<-4(k^2+bk)$. There exist $a_\star>0$ and $s_* >0$ such that for any $\tilde s \in [0, s_*)$, there exists a unique solution $(s, r, \gamma) (a)$ of \eqref{eq4.17}, or equivalently a unique solution $W(a)$ of \eqref{eq4.8} up to the rotation, such that its domain contains $[0, a_\star]$ and $s(a_\star) = \tilde s$ and $s(0) =0$. Moreover, this solution satisfies $s(a) = O(a^\kappa)$ and $|W_a(a)| = O(a^{\kappa-1})|$ for $|a|<<1$, where $\kappa = \sqrt{-(k^2+bk)-\frac{c^2}{4}}>0$. 
\end{prop}

The $\sigma$ variable can be solved explicitly in terms of $s$. In fact 

\begin{lemma} \label{L:sigma} 
Let $W(a)$ be a solution of \eqref{eq4.8} such that $s(0)=0$ and $\lim_{a\to 0+} a |w_a(a)|=0$, then it  satisfies 
\begin{equation}\label{eq4.23}
\sigma=-\frac{cF}{a\Gamma} \text{ and } |W_a|^2 = s_a^2 + \sigma^2 = s_a^2 + \frac{c^2F^2}{a^2\Gamma^2}
\end{equation}
\end{lemma}

\begin{proof} 
Multiplying \eqref{eq4.14} by $a\Gamma(s)$, one has
\[a\Gamma\sigma_a+a\sigma\Gamma_s s_a+\sigma\Gamma=-c\Gamma s_a,\]
which can be written as
\[(a\Gamma\sigma)_a=(-cF)_a.\]
Integrating the above equation with respect to $a$ and using the fact $a|W_a|\to 0$ as $a\to0$ and $F(0)=0$, the proof of the lemma is complete. 
\end{proof}

\begin{rem} 
Recall $\beta$ is the angular coordinate on the target surface. For any solution given by  Proposition \ref{P:local}, $\lim_{a\to 0+} a \beta_a = \lim_{a\to 0+} \frac {a\sigma}\Gamma = -\frac c2$. Therefore $\beta (a) = O(\log a)$ diverges as $a\to 0+$ if $c \ne 0$. It mean that these solutions converge to $s=0$ as $a \to 0+$ in a spiral fashion. 
\end{rem}

According to Lemma \ref{L:sigma}, for solutions of \eqref{eq4.8}  obtained in Proposition \ref{P:local} we have
\begin{equation}\label{eq4.24}
s_{aa} + \frac 1a s_a -\mu\Gamma + \frac{1}{a^2}(G+\frac{c^2F^2}{2\Gamma^2})_s =0.
\end{equation}

\begin{cor} \label{C:pole} 
Suppose $c\ne 0$ and $s_0<\infty$ where $s_0$ was given in the definition of $\Gamma$ in Section \ref{S:coord}. Let $W(a)$, $a\in (-a_1, a_2)$, $a_{1,2} \in (0, \infty]$, be a solution of \eqref{eq4.8} obtained in Proposition \ref{P:local}. Then $s(a) \in (-s_0, s_0)$ for all $a\in (-a_1, a_2)$. 
\end{cor}

In fact, the corollary simply follows from $s_a^2 + \frac {c^2 F^2}{a^2 \Gamma^2} = |W_a|^2 <\infty$ at any $a\in (-a_1, a_2)$. 

The assumption $s_0<\infty$ is equivalent to compactness of the surface of revolution. In the case of ${\mathbb S^2}$, the above corollary means those solutions (starting at the north pole) would never reach the south poles for any finite $a$ in the interval of their existence. 

\subsection{Global existence of solutions to \eqref{eq4.8}} \label{SS:global1}
Before we consider the asymptotic behavior as $a\to+\infty$ of solutions to \eqref{eq4.7} or equivalently  \eqref{eq4.8}, we first establish certain general global existence of solutions. 

\begin{lemma} \label{L:global}
\begin{enumerate}
\item If the target manifold satisfies that there exists $C>0$ such that 
\[
|F| \le C (1+|s|),
\]
then any solution $W(a)$ of \eqref{eq4.8} with value given at $a_0>0$ is defined for all $a\ge a_0$. 
\item Suppose $\mu <0$. For any $a_0>0$ and $S>0$, there exists $\delta>0$ such that if $W(a)$ is a solution of \eqref{eq4.8} and 
\[
\frac12|W_a(a_0) |^2-\mu F\big(W(a_0)\big) \le \delta,
\]
then $W(a)$ exists for all $a\ge a_0$, 
\[
\frac 12 |W_a|^2 -\mu F \le S, \quad \forall a\ge a_0
\]
and
\[
\lim_{a\to\infty} \big(\frac 12 |W_a|^2 -\mu F\big) \ \text{ and } \  \int_{a_0}^\infty\frac{1}{a'}|W_a|^2\ da'=\int_{a_0}^\infty\frac {1}{a'} (s_a^2+\sigma^2)\ da' \ \text{ exist}.
\] 
\end{enumerate} 
\end{lemma} 

\begin{rem} 
1.) Note the first conclusion of the lemma applies to any smooth compact manifold of revolution.\\ 
2.)  Recall that, as defined in Section \ref{S:coord}, $\Gamma(s)$ is $2s_0$-periodic in $s$ and $F$ achieves its maximum at $s=(2n+1) s_0$, an odd multiple of $s_0 \in (0, \infty]$. In particular $s_0 =\pi$ for ${\mathbb S^2}$ and $\infty$ for ${\mathbb H^2}$.  If in the second part of the lemma $S <-\mu \sup_{s\in [0, s_0]} F(s)$, then $\mu<0$ and $\frac 12 |W_a|^2 -\mu F \le S$ imply that the solution satisfies $|s(a)| \le s_1 \triangleq \inf_{s>0} \{F(s)>-\frac S\mu\} \in (0, s_0)$ for all $a\ge a_0$. \\
3.) If $\mu>0$ and the target surface is ${\mathbb H^2}$, it is not clear if there exist an open set of initial data whose solutions do not blow up at finite $a>0$. 
\end{rem}

\begin{proof} 
To prove the first statement, let $a_1>a_0$ and we will obtain a priori estimates on $s(a)$ and $|W_a(a)|$ on $[a_0, a_1]$. In fact, for any $a \in [a_0, a_1]$, on the one hand, 
\[s(a)=s(a_0)+\int_a^{a_0} s_a(a')\ da'\]
implies
\[
s(a)^2 \le 2s(a_0)^2 + 2 \big( \int_{a_0}^a |W_a (a')| da' \big)^2 \le C'\big(1+ \int_{a_0}^a |W_a (a')|^2 da'\big)
\]
where the constant $C'$ depends on $a_0, \ a_1, \ s(a_0)$, but independent of $a$. On the other hand, from \eqref{eq4.11} and \eqref{eq4.12} and the assumptions on $F$, we obtain, for $a \in [a_0, a_1]$, 
\[\begin{split}
|W_a(a)|^2 =& \frac 2{a^2} \big( \frac {a_0^2}2 |W_a(a_0)|^2 +G(a_0)-\mu a_0^2 F(a_0)\\
&-\frac {k^2}2 \Gamma^2 -b k F + k^2 F^2 + \mu a^2 F - 2\mu \int_{a_0}^a a'F da'\big)  \\
\le& C' \big(1+ |s(a)|^2 + \int_{a_0}^a s(a')^2 da' \big) 
\end{split}\]
where again the constant $C'$ depends on $a_0, \ a_1, \ |W_a(a_0)|$ and $|\Gamma(a_0)|$ but independent of $a$. The above inequalities imply 
\[
C_1 s(a)^2 + |W_a (a)|^2 \le C_2 \big(1+ \int_{a_0}^a C_1 s(a')^2 + |W_a (a')|^2 da'\big)
\]
where $C_{1,2}$ are independent of $a$. The Gronwall inequality implies the estimates of $s(a)$ and $|W_a(a)|$ on $[a_0, a_1]$. As $a_1$ is arbitrary, we obtain the global existence of $W(a)$ for $a\ge a_0$. 

To prove the second statement where $\mu<0$, without loss of generality, we may assume $S <-\mu \sup_{s\in [0, s_0]} F(s)$ where $s_0$ was given in the definition of $\Gamma$ in Section \ref{S:coord}. Let 
\[
s_1 = \inf_{s>0} \{F(s)>-\frac S\mu\} \in (0, s_0), \quad C_0= \sup_{s \in [-s_1, s_1]} |\frac {G_s}{\sqrt{-\mu F}}| < \infty. 
\]
where $C_0<\infty$ since near $s=0$, both $G_s$ and $\sqrt{F}$ are of order $O(s)$. Let
\[
a_1 = \sup \{a\ge a_0 \ :\ \frac12|W_a|^2-\mu F\le S \text{ on } [a_0, a]\}. 
\]
Clearly $a_1>a_0$ if we choose $\delta <S$. Multiplying \eqref{eq4.8} by $W_a$ we obtain, for $a \in [a_0, a_1]$,
\begin{equation}\label{eq4.26}\begin{aligned}
\partial_a(\frac12|W_a|^2-\mu F)&=-\frac1a|W_a|^2- \frac 1{a^2} G_s s_a
\le -\frac1{a}|W_a|^2 + \frac {C_0}{a^2} \sqrt{-\mu F} s_a \\
& \le  -\frac1{a}|W_a|^2 + \frac {C_0}{a^2}(\frac12|W_a|^2-\mu F).
\end{aligned}
\end{equation}
Taking $\delta = S e^{-\frac {C_0}{a_0}}$, dropping the good term $-\frac1{a}|W_a|^2$ in the above, and applying Gronwall inequality, we obtain $a_1 = \infty$ and $\frac12|W_a|^2-\mu F \le S$ for all $a\ge a_0$. In particular, we have for all $a\geq a_0$,
\[|s(a)|<s_1\ , \ |s_a(a)|\leq\sqrt{2S}\ , \ F(s(a))\leq -\frac S\mu,\]
which implies
\[|G_s s_a|\leq C_0 S\ , \ \text{for}\ a\geq a_0.\]
Finally let 
\[
g(a) = \int_{a_0}^a  \frac 1{a^2} G_s s_a \ da'.
\]
Clearly $g$ is bounded and $\lim_{a\to \infty} g(a)$ exists since $\frac 1{a^2} |G_s s_a| \le \frac {C_0 S}{a^2}$ is integrable. The first equality in \eqref{eq4.26} implies $\frac12|W_a|^2-\mu F + g$ is bounded in the below by $S-\frac {C_0S}{a_0}$ and is decreasing and thus we obtain the last convergence results claimed in the lemma.
\end{proof}

In the following, we split our discussion into cases $\mu<0$ and $\mu>0$.

\subsection{Asymptotic behavior of solutions near $a=\infty$ assuming $c^2 < -4(b k + k^2)$ and $\mu<0$} \label{SS:negative-mu}

We will analyze the asymptotic behavior of solutions given by Proposition \ref{P:local} under the assumption
\[
\text{({\bf A1})}\ \ \ \ \ \ S \triangleq \limsup_{a\to \infty} |s(a)| < s_0.
\]
Lemma \ref{L:global} ensures that all solutions  given by Proposition \ref{P:local} sufficiently close to $W(a) \equiv 0$ satisfy the above condition. In fact, the following lemma indicates that a larger class of solutions satisfy this assumption. 

\begin{lemma} \label{L:-} 
Suppose $\mu<0$ and $c^2 < -4(k^2+bk)$. Let $W(a)$, $a\in (0, a_0]$, be a nontrivial solution of \eqref{eq4.8} satisfying $s(0)=0$ and $\lim_{a\to 0+} a |W_a(a)| =0$. Then for all $a \in (0, a_0]$ 
\[
|s(a)| \le \bar s \triangleq \sup \{s \ : \ \frac {c^2 F^2}{2 \Gamma^2} + G \le 0 \text{ on } [0, s]\}. 
\]
\end{lemma}

\begin{rem} 
1.) Note that if $\bar s <s_0$ and a solution $W(a)$ of \eqref{eq4.8} satisfies $s(0)=0$ and $\lim_{a\to 0+} a|W_a(a)|=0$, then the above lemma implies that $W(a)$ exists for all $a\in {\mathbb R}$ and satisfies {\bf (A1)}. \\
2.) If $c\ne 0$ and $s_0<\infty$ (i.e. the target manifold is compact), since
\[\frac{c^2F^2(a_0)}{2\Gamma^2(a_0)}+G(a_0)=+\infty,\]
then clearly $\bar s<s_0$.\\
3.) If $c=0$ or $s_0 =\infty$, some manifolds may still satisfy $\bar s<s_0$.   
\end{rem}

\begin{proof}
Note that a simple Taylor expansion at $s=0$ implies $\bar s>0$ under the assumption of the lemma. Moreover, we claim $\bar s<s_0$ or $\bar s=+\infty$. For $s_0=+\infty$, the claim is trivial. For $s_0<+\infty$ and $c\neq0$, we have $\bar s<s_0$. Therefore, we only need to consider the case where $s_0<\infty$ and $c=0$. In this case, we note $G$ is even about $s=0$ and $2s_0$-periodic and thus if $\bar s\geq s_0$, we must have $\bar s=+\infty$. 

To complete the proof of the lemma, we assume $\bar s< s_0$ without loss of generality. Let 
\[
a_1 = \sup \{a>0\ : \ |s| \le \bar s \text{ on } [0, a]\} >0.   
\]
If $a_1 < a_0$, then $|s(a_1)|=\bar s$ and $F(s(a_1)) \ge F(s(a))$ for all $a\in [0, a_1]$ due to the monotonicity of $F$ on $[0, s_0)$. Multiplying \eqref{eq4.24} by $a^2 s_a$ and integrating from $0$ to $a_1$, we obtain
\[
\frac {a_1^2}2 s_a^2 -\lim_{a\to0}(\frac{a^2}{2}s_a^2(a))= \int_0^{a_1}a^2(\mu F)_a\ da+\left(\frac {c^2 F^2}{2 \Gamma^2}   + G\right)(s(a_1)).
\]
Lemma \ref{L1} shows the second term on the left hand side vanishes. Then we integrate the first term on the right hand side by parts to obtain
\[\frac {a_1^2}2 s_a^2 = \int_0^{a_1}2a\mu(F(s(a_1))-F(s(a)))\ da+\left(\frac {c^2 F^2}{2 \Gamma^2}   + G\right)(s(a_1))<0,\]
which is impossible.
\end{proof}

Let us returned to the study of the asymptotic properties of solutions given by  Proposition \ref{P:local}  under assumption {\bf (A1)}. Let 
\[
s_1 = \frac 12 S +\frac 12 \min\{s_0,  1\} \in [0, s_0).
\]
Due to remark 2.) following Lemma \ref{L:global}, {\bf (A1)} implies that there exists $a_1>0$ such that  
\begin{equation}\label{eq4.25}
|s(a)| \le s_1<s_0, \text{ for } a \ge a_1\text{ and } |s(a)| < s_0 \text{ for all }  a\ge 0.
\end{equation}

Let $w=\sqrt a s$ and plug into \eqref{eq4.24}, we obtain
\begin{equation}\label{eq4.29}
w_{aa}+\frac{1}{4a^2}w-\sqrt a\mu\Gamma(\frac {w}{\sqrt a})=-\frac{G_s(\frac{w}{\sqrt a})}{a^{\frac32}}+\frac{c^2\tilde G(\frac{w}{\sqrt a})}{a^{\frac 32}}
\end{equation}
where $\tilde G(s)=-\big(\frac {F^2}{2 \Gamma^2}\big)_s=\frac{\Gamma_sF^2}{\Gamma^3}-\frac{F}{\Gamma}$.
Multiplying the above equation by $w_a$, we obtain
\begin{equation}\label{eq4.30}
\begin{aligned}
\p_a E(a)\triangleq& \p_a[\frac 12 w_a^2+\frac{1}{8a^2}w^2-\mu a F(a^{-\frac 12}w)]\\
=&-\frac{1}{4a^3}w^2+\mu\Gamma(a^{-\frac 12}w)(\frac 12a^{-\frac 12}w)-\mu F(a^{-\frac 12}w)\\
&+a^{-\frac 32}\big(-G_s(\frac{w}{\sqrt a})+c^2\tilde G(\frac{w}{\sqrt a})\big)w_a.
\end{aligned}
\end{equation}
Since \eqref{eq4.25} implies $s(a)$ for $a>0$ is uniformly bounded, the last terms of \eqref{eq4.30} can estimated by
\begin{equation}\label{eq4.32}
|a^{-\frac 32}(G_s(\frac{w}{\sqrt a})+c^2\tilde G(\frac{w}{\sqrt a}))w_a|\leq \min\{\frac{1}{4a}w_a^2+\frac{C}{a^2},\frac{C}{a^2}|w||w_a|\},
\end{equation}
by controlling $|G_s|$ and $|\tilde G|$ by $C$ or $C|s|$ respectively, the latter of which holds as 
$G_{s}(0)=\tilde G(0)=0$. 
We also claim that there exists $d\in[0,1)$ such that
\begin{equation}\label{eq4.31}
F(s(a)) - \frac 12 s(a) \Gamma(s(a)) \le d F(s(a)) \quad \forall a>0. 
\end{equation}
In fact, on the one hand, \eqref{eq4.25} implies $s(a) \Gamma(s(a))$is uniformly bounded away from $0$ except when $s(a)$ is close to $0$. On the other hand, both $F(s)$ and $s \Gamma(s)$ are of order $O(s^2)$ for $|s|<<1$. Therefore $\frac {F(s(a))}{s(a)\Gamma(s(a))}$ is bounded for all $a>0$ which implies \eqref{eq4.31}. 
By \eqref{eq4.30}, \eqref{eq4.31} and \eqref{eq4.32}, we have
\begin{equation}\label{eq4.33}
\begin{aligned}
\p_a E(a)&\leq \mu\Gamma(a^{-\frac 12}w)(\frac 12a^{-\frac 12}w)-\mu F(a^{-\frac 12}w)+\frac{1}{4a}w_a^2+O(a^{-2})\\
&\leq \frac da E(a)+O(a^{-2}),
\end{aligned}
\end{equation}
where $d\in[0,1)$. Consequently, $E(a)=O(a^d)$. Again by \eqref{eq4.25},
\[\frac{-\mu}{C}w^2<-\mu a F(a^{-\frac 12}w),\]
and thus
\begin{equation}\label{eq4.34}
|w_a|+|w|=O(a^{\frac d2}),
\end{equation}
which implies \[s(\infty)=0.\]

A more precise asymptotic formula for $s$ as $a\to\infty$ can be obtained as follows. By the Taylor's expansions at $s=0$,
\begin{equation}\label{eq4.35}
\begin{aligned}
&|\mu\Gamma(a^{-\frac 12}w)(\frac 12a^{-\frac 12}w)-\mu F(a^{-\frac 12}w)|\\
=&|\mu|\big|(s+O(s^3))\frac 12 s-\frac 12 s^2-O(s^4)\big|=O(\frac{|w|^4}{a^2}).
\end{aligned}
\end{equation}
From \eqref{eq4.30}, and \eqref{eq4.34}, we obtain for $a>>1$, there exists $C>0$ such that
\begin{equation}\label{eq4.36}
\p_a E\leq C\big(a^{d-3}+(\frac{w}{\sqrt a})^4+a^{d-2}\big),
\end{equation}
where the second term on the right hand side is given by \eqref{eq4.35} and the third term is given by \eqref{eq4.32}. Hence,
\[\p_a E \leq C (a^{d-2} + \frac {w^2}{a^2} w^2) \leq Ca^{d-2}(E+1).\]
This shows
\begin{equation}\label{eq4.37}
E=O(1)=|w|+|w_a|.
\end{equation} 
Therefore, in \eqref{eq4.36} $d$ can be taken as $0$ and it gives
\[|\p_a E|\leq \frac{C}{a^2}.\]
Meanwhile we also deduce from \eqref{eq4.30} along with \eqref{eq4.32} and \eqref{eq4.35} that
\[\begin{aligned}
\p_a E\geq-\frac{1}{4a^3}w^2-C\frac{w^2}{a^2}|w|^2-\frac{C}{a^2}(|w|^2+|w_a|^2)\geq-\frac{C}{a^2}E,
\end{aligned}\]
which implies $|\p_a(\log E)|\le \frac C{a^2}$ is integrable as $a \to \infty$. Hence unless $w(a) \equiv0$ we have 
\begin{equation} \label{eq4.35.1} 
E_\infty \triangleq \lim_{a\to \infty} E(a) \in (0, \infty) \text{ and } E(a) = E_\infty + O(\frac 1a).  
\end{equation}

Finally, we let
\[w=r\cos \theta\ , \ w_a=(-\mu)^{\frac 12}r\sin \theta.\]
Clearly 
\[
E_\infty + O(\frac 1a) = E(a) = -\frac \mu2 r^2 + \frac 1{8a^2} r^2 \cos^2 \theta - \mu \big( a F(a^{-\frac 12} r \cos \theta) - \frac 12 r^2 \cos^2 \theta\big) 
\]
which along with the evenness of $F$ implies 
\begin{equation} \label{eq4.35.2} 
r = r_\infty + O(\frac 1a) >0 \text{ where } r_\infty = \sqrt{-\frac {2E_\infty}\mu}.  
\end{equation} 
One may compute using the oddness of $\Gamma$ and \eqref{eq4.35.2}
\begin{equation}\label{eq4.38.a}
\begin{aligned}
r_a=&- (-\mu a)^{\frac 12}\big( \Gamma (a^{-\frac 12}r\cos\theta)-a^{-\frac 12}r\cos\theta\big)\sin\theta \\
&\qquad - (-\mu)^{-\frac 12} \big(\frac 1{4a^2} r \cos \theta + a^{-\frac 32} (G_s - c^2 \tilde G) (a^{-\frac 12}r\cos\theta) \big) \sin \theta = O(\frac 1a) 
\end{aligned}
\end{equation}
and 
\begin{equation}\label{eq4.38.b}
\begin{aligned}
\theta_a=&-(-\mu)^{\frac 12} \big( \sin^2 \theta + \frac 1r a^{\frac 12}\cos \theta\Gamma (a^{-\frac 12}r\cos\theta)  \big) \\
&\qquad -(-\mu)^{-\frac 1{2}}\frac{\cos\theta}r  \big(\frac 1{4a^2} r \cos \theta + a^{-\frac 32} (G_s - c^2 \tilde G) (a^{-\frac 12}r\cos\theta) \big)\\
=& -(-\mu)^{\frac 12} - \frac{(-\mu)^{-\frac 12} \Gamma'''(0) E_\infty}{3a}(\frac 38+\frac 12\cos{2\theta}+\frac 18\cos{4\theta})+O(\frac{1}{a^2}).
\end{aligned}
\end{equation}
Note that \eqref{eq4.38.a} and \eqref{eq4.38.b} and some tedious but straight forward calculation of $\theta_{aa}$ imply $\theta_{aa} = O(\frac 1a)$,
and thus for $n=2,4$, 
\[\int_{a_0}^a\frac{\cos{n \theta}}{a'}\ da'=\frac{\sin{n \theta}}{n a'\theta_a}\Big|_{a_0}^a+\frac 1n\int_{a_0}^a\frac{\sin{n \theta} (\theta_a+a'\theta_{aa})}{(a'\theta_a)^2}\ da'\]
which converges at the rate $O(\frac 1a)$. Therefore, for some $\theta_0 \in {\mathbb R}$
\begin{equation}\label{eq4.39}
\theta(a)=\theta_0-(-\mu)^{\frac 12}a-\frac{(-\mu)^{-\frac 12} \Gamma'''(0) E_\infty}{8}\log a+O(a^{-1}).
\end{equation}
In summary, we have proved 

\begin{prop} \label{P:-}
Suppose $\mu <0$ and $c^2 < -4(b k + k^2)$. Let $W(a)$, defined for sufficiently large $a>0$, be a solution of \eqref{eq4.8} satisfying  condition {\bf (A1)}. Then there exist constants $\theta_0 \in \mathbb R$ and $E_\infty>0$ such that as $a\to +\infty$
\[\begin{aligned}
&s(a)=\sqrt{\frac{-2 E_\infty}{a\mu}}
\cos{\big(\theta_0-(-\mu)^{\frac 12}a-\frac{(-\mu)^{-\frac 12}\Gamma'''(0) E_\infty}{8}\log a+O(a^{-1})\big)}+O(a^{-\frac 32}),\\
& s_a(a)=\sqrt{\frac{2 E_\infty}{a}}
\sin{(\theta_0-(-\mu)^{\frac 12}a-\frac{(-\mu)^{-\frac 12}\Gamma'''(0)E_\infty}{8}\log a+O(a^{-1}))}+O(a^{-\frac 32}),\\
&\sigma(a)=-c \sqrt{\frac{-E_\infty}{2a^3\mu}}
\cos{\big(\theta_0-(-\mu)^{\frac 12}a-\frac{(-\mu)^{-\frac 12}\Gamma'''(0)E_\infty}{8}\log a+O(a^{-1})\big)}+O(a^{-\frac 52}) \\
&\qquad = -\frac{cF(s(a))}{a \Gamma(s(a))}.
\end{aligned}\]
\end{prop} 

In particular, the asymptotic form of $\sigma$ implies $\beta_a=\frac\sigma\Gamma=O(\frac 1a)$ and is oscillatory. A simple argument based on integration by parts implies that $W(a)$ converges to the rotation center $s=0$ (at the tangential level) along one single geodesic as $a \to \infty$. 

\subsection {Asymptotic behavior of solutions targeted on compact surfaces near $a=\infty$ assuming $0<c^2 < -4(b k + k^2)$ and $\mu>0$} \label{SS:positive-mu}

In this case of $\mu>0$, it is not clear if the solutions obtained in Proposition \ref{P:local} exist globally if the target surface is like ${\mathbb H^2}$. Instead, we will focus on {\it compact} target surface, i.e. $s_0<\infty$. In addition to the global existence of solutions given by Lemma \ref{L:global}, we have 

\begin{lemma} \label{L:+} 
Suppose $c\ne 0$, $s_0< \infty$, and $\mu>0$. Let $W(a)$, $a>0$, be a nontrivial solution of \eqref{eq4.8} such that $s(0)=0$ and $\lim_{a\to 0+} a|W_a(a)| =0$. Then $\Gamma (s(a))\ne 0$ for all $a>0$ and $\liminf_{a\to +\infty} |s(a)| >0$.
\end{lemma} 

\begin{proof} 
On the one hand, Corollary \ref{C:pole} implies $s(a) \ne (2n+1) s_0$ for any $a>0$. On the other hand,   from the assumptions of $W(a)$ at $a=0$ and equalities \eqref{AssumAtZero} and \eqref{eq4.11}, we have for any $a>0$ 
\[
0> -2 \mu \int_0^a a' F(s(a')) da' = \frac {a^2}2 |W_a(a)|^2 + G(s(a)) - \mu a^2 F(s(a))
\]
and thus $s(a) \ne 2n s_0$ for any $a>0$. Without loss of generality, we may assume $s(a) \in (0, s_0)$ for all $a>0$. 

To prove the lemma, we argue by contradiction. Assume $\liminf_{a\to +\infty} s(a) =0$. Hence there exists a sequence $0< a_0 < a_1< \ldots  \to \infty$ such that $s_0> s(a_0) > s(a_1) > s(a_2) > \ldots \to 0$ and $s(a_n) = \min_{a\in [a_0, a_n]}  s(a)$. Therefore, 
\[\begin{split}
0>& -2 \mu \int_0^{a_0} a F(s(a)) da = \frac {a_0^2}2 |W_a(a_0)|^2 + G(s(a_0)) - \mu a_0^2 F(s(a_0))  \\
=& \frac {a_n^2}2 |W_a(a_n)|^2 + G(s(a_n)) - \mu a_0^2 F(s(a_n)) + 2\mu \int_{a_0}^{a_n} a \big(F(s(a)) - F(s(a_n))\big) da.    
\end{split}\]
On the right side of the above, due to the choice of the sequence $\{a_n\}$, $G(s(a_n))$ and $F(s(a_n))$ converge to $0$ as $n \to \infty$ and the other two terms are non-negative. This is impossible and the lemma follows. 
\end{proof}

In the spirit of Lemma \ref{L:+}, without loss of generality, we assume 
\begin{equation} \label{eq4.H'}
s(a) \in (0, s_0)
\end{equation} 
in the rest of this subsection. We shall study the behavior of solutions given by Proposition \ref{P:local} as $a\to +\infty$ in a procedure much that in the last subsection. As Lemma \ref{L:+} hints that $s(a) \to s_0$ as $a\to +\infty$, let 
\[\begin{split}
&\tilde s(a)= s_0 - s(a) \in (0, s_0), \quad F_1 (\tilde s) = F(s_0) - F(s_0 - \tilde s)\ge 0, \quad \Gamma_1 (\tilde s) = \Gamma(s_0 - \tilde s), \\
& G_1(\tilde s) = \big(G + \frac {c^2 F^2}{2\Gamma^2} - b k F(s_0) + k^2 F(s_0)^2 + A\big)|_{s_0 - \tilde s} \ge A >0, \quad \forall \tilde s \in (0, s_0)
\end{split}\]
where we added the positive constant $A$ and those involving $F(s_0)$ to make $G_1$ have positive lower bound without changing $G_{1\tilde s}$. We rewrite equation \eqref{eq4.24}
\begin{equation}  \label{eq4.24.1}
\tilde s_{aa} + \frac 1a \tilde s_a + \mu  \Gamma_1 (\tilde s) + \frac{1}{a^2}G_{1\tilde s} (\tilde s)=0.
\end{equation} 

Let $w=\sqrt{a} \tilde s$ and we have 
\begin{equation}  \label{eq4.41}
w_{aa}+\frac{1}{4a^2}w + \mu\sqrt a\Gamma_1 (\frac {w}{\sqrt a}) + a^{-\frac32} G_{1\tilde s} (\frac{w}{\sqrt a}) =0.
\end{equation} 
Let 
\[
E(a) = \frac 12w_a^2+\frac{1}{8a^2}w^2+\mu a F_1 (\frac {w}{\sqrt a}) + \frac 1a G_1 (\frac {w}{\sqrt a})
\]
and one may compute 
\begin{equation} \label{eq4.42} 
\p_a E = -\frac{1}{4a^3}w^2 + \Big(\mu \big(F_1(\tilde s) - \frac 12 \tilde s \Gamma_1 (\tilde s)\big)
- \frac 1{a^2} \big( G_1 (\tilde s) + \frac 12 \tilde s G_{1s} (\tilde s) \big) \Big)|_{\tilde s=  \frac {w}{\sqrt a}}. 
\end{equation} 
Since $\Gamma(s_0) =0$ and $\Gamma_s(s_0) =-1$, we have both $F_1(\tilde s)$ and $\frac 12\tilde s \Gamma_1(\tilde s)$ are like $\frac 12 \tilde s^2 + O(\tilde s^4)$ for $|\tilde s|<<1$. Moreover Lemma \ref{L:+} implies that $\tilde s \Gamma_1 (\tilde s)$ is away from $0$ for $a \ge 1$. Therefore, there exists $d\in[0,1)$ and $C_1>0$ such that
\begin{equation} \label{eq4.43} 
F_1(\tilde s) - \frac 12 \mu \tilde s \Gamma_1 (\tilde s) \le d F_1(\tilde s) \text{ and } |F_1(\tilde s) - \frac 12  \tilde s \Gamma_1 (\tilde s)|\le C_1 \tilde s^2 F_1(\tilde s)  \quad \forall a>1.
\end{equation}
To handle the terms involving $G_1$, which is even in $\tilde s$, we use the Taylor expansions of $F$ and $\Gamma$ at $s_0$ to obtain, for $|\tilde s| <<1$, 
\begin{equation} \label{eq4.43.1}\begin{split} 
G_1(\tilde s) =& \frac {c^2 F(s_0)^2}{2\tilde s^2}+ \frac {c^2 F(s_0)}2 \big(\frac 13 F(s_0) \Gamma'''(s_0) -1\big) +A + O(\tilde s^2) \\
\frac 12 \tilde s G_{1\tilde s} (\tilde s) =& -\frac {c^2 F(s_0)^2}{2\tilde s^2}  + O(\tilde s^2).
\end{split}\end{equation}
Since $G_1$ is smooth away from $s_0$, therefore by choosing $A$ reasonably large we have 
\[
C_2 \tilde s^2 G_1 (\tilde s) > G_1 (\tilde s) + \frac 12 \tilde s G_{1s} (\tilde s) >0
\]
for some constant $C_2>0$. It along with \eqref{eq4.43} implies 
\[
\p_a E \le \frac da E. 
\]

Based on \eqref{eq4.43} and proceeding much as in \eqref{eq4.30}--\eqref{eq4.35.1}, one can show
\begin{equation} \label{eq4.43.2}
\lim_{a\to 0}E(a)=E_\infty>0, \; \p_a E(a) = O(\frac 1{a^2}), \; E(a) =E_\infty + O(\frac 1a), \;  |w|+|w_a|=O(1).
\end{equation} 
After the above convergence is derived, we may set 
\[
A=0
\]
in the definition in $G_1$ without changing anything. Using the above Taylor's expansions of $G_1$ and $F_1$, we also obtain 
\begin{equation}\label{eq4.44}
E (a) = \frac 12 w_a^2+\frac 12 \mu w^2  + \frac{c^2 F(s_0)^2}{2w^2} - \frac {\mu \Gamma'''(s_0)}{24 a} w^4 +O(\frac 1{a^2}) \; \Rightarrow \; E_\infty \ge\sqrt \mu |c| F(s_0).
\end{equation}

As in the previous case, we shall derive more detailed expansion of $w(a)$. In fact we need further distinguish two cases, namely, $E_\infty=\sqrt \mu |c| F(s_0)$ and $E_\infty>\sqrt \mu |c| F(s_0)$. 

In the case of $E_\infty =\sqrt \mu |c| F(s_0)$, \eqref{eq4.44} immediately implies 
\[
w= \mu^{-\frac 14} |c|^{\frac 12} F(s_0)^\frac 12 + O(\frac 1a),  \quad 
|w_a|=O(a^{-\frac 12}).\]

For $E_\infty>\sqrt \mu |c| F(s_0)$, The leading order of $w_a$ would be of $O(a^{-\frac 12})$. In fact, multiplying \eqref{eq4.44} by $w^2$ we obtain 
\[
w^2w_a^2 + \mu (w^2 - \mu^{-1} E_\infty)^2 + c^2 F(s_0)^2 - \mu^{-1} E_\infty^2 = O(a^{-1})
\]
which motivates us to introduce 
\begin{equation} \label{eq4.44.1}
w_aw=\mu^{\frac 12}r\cos\theta\ \ ,\ \ w^2 - \mu^{-1} {E_\infty} = r\sin\theta.
\end{equation}
Clearly 
\[r=r_\infty +O(a^{-1}) \text{ where } r_\infty =\mu^{-1} {\sqrt{E_\infty^2- \mu c^2 F(s_0)^2}}.\]
One may compute
\[\begin{aligned}
(\mu^{\frac 12}r)\theta_a=& \mu^{\frac 12}(r\sin \theta)_a\cos\theta-(\mu^{\frac 12} r\cos \theta)_a\sin\theta\\
=&\mu^{\frac 12}(w^2 - \mu^{-1} {E_\infty})_a\cos\theta-(w_aw)_a\sin\theta\\
=&\mu^{\frac 12}(2\mu^{\frac12}r\cos\theta)\cos\theta-(w_{aa}w+w_a^2)\sin\theta\\
=&2\mu r\cos^2\theta + \big(\mu \sqrt a w \Gamma_1 (\frac w{\sqrt a}) + a^{-\frac 32} w G_{1\tilde s} (\frac w{\sqrt a}) - w_a^2\big) \sin \theta 
+O(\frac{1}{a^2})
\end{aligned}\] 
where we used \eqref{eq4.41} in the last equality. From \eqref{eq4.43.1} and the Taylor expansion of $F_1$, it is straight forward to compute 
\[
\mu \sqrt a w \Gamma_1 (\frac w{\sqrt a}) + a^{-\frac 32} w G_{1\tilde s} (\frac w{\sqrt a}) = \mu w^2 - \frac {c^2 F(s_0)^2}{w^2} - \frac {\mu \Gamma'''(s_0)}{6a} w^4 + O(a^{-2}).  
\]
It implies 
\begin{equation}\label{eq4.45} \begin{split}
\theta_a=& 2 \mu^{\frac 12} \cos^2 \theta + \frac 2{\sqrt \mu r}  (\mu w^2 - E(a) - \frac {\mu \Gamma'''(s_0)}{8a} w^4 )\sin \theta + O(a^{-2}) \\
=& 2\mu^{\frac 12} + \frac 2{\sqrt  \mu r} \big(E_\infty - E(a)  - \frac {\mu \Gamma'''(s_0)}{8a} (r_\infty \sin \theta + \mu^{-1} E_\infty)^2 \big) \sin \theta +O(\frac 1{a^2}).
\end{split}\end{equation}
Similar (but tedious)  calculation shows
\begin{equation}\label{eq4.46}
 r_a= O(\frac 1a), 
 \quad \theta_{aa} = O(\frac 1a). 
\end{equation}
Integrating the right hand side of \eqref{eq4.45} by parts much as in the previous subsection, using the faster decay of $\p_a E(a) = O(a^{-2})$ and $r_a$ and $\theta_{aa}$, we obtain
\[\theta(a)=\theta_0+2\mu^{\frac 12}a+O(\frac 1a)\]
for some $\theta_0 \in \mathbb R$. Therefore,
\[\begin{aligned}
w=&(\mu^{-1} E_\infty+r\sin\theta)^{\frac 12}=\mu^{-\frac 12} \sqrt {E_\infty+ (E_\infty^2-\mu c^2 F(s_0)^2)^{\frac 12}  \sin (\theta_0 + 2 \mu^{\frac 12} a) }+O(\frac 1a),\\
w_a=&\frac{\mu^{\frac 12}r\cos\theta}w= \sqrt {\frac {E_\infty^2-\mu c^2 F(s_0)^2}{E_\infty+ (E_\infty^2-\mu c^2 F(s_0)^2)^{\frac 12}  \sin (\theta_0 + 2 \mu^{\frac 12} a)}} \cos (\theta_0 + 2 \mu^{\frac 12} a)+O(\frac 1a).
\end{aligned}\]
Consequently, we obtain 

\begin{prop} \label{P:+}
Suppose $s_0<\infty$, $\mu>0$, and $0< c^2 < -4(b k + k^2)$. Let $W(a)$, $a>>1$, be a solution of \eqref{eq4.8} such that $s(a) \in (0, s_0)$ and $\liminf_{a\to +\infty} s(a)>0$. Then we have either 
\[\begin{aligned}
&s(a)= s_0 - \mu^{-\frac 14} |c|^{\frac 12} F(s_0)^\frac 12 a^{-\frac 12} 
+O(a^{-\frac 32}), \quad s_a(a)=O(a^{-1}),\\
&\sigma(a)= -\frac {cF}{a\Gamma} 
= - \frac {\mu^{\frac 14} c F(s_0)^{\frac 12}}{\sqrt{|c| a}} + O(a^{-\frac 32}).
\end{aligned}\]
or there exist constants $E_\infty > \sqrt \mu |c| F(s_0)$ and $\theta_0$ such that 
\[\begin{aligned}
&s(a)= s_0 - \frac 1{\sqrt {\mu a}} \sqrt {E_\infty+ (E_\infty^2-\mu c^2 F(s_0)^2)^{\frac 12}  \sin (\theta_0 + 2 \mu^{\frac 12} a) }+O(\frac 1a), \\
&s_a(a)= -a^{-\frac 12} \sqrt {\frac {E_\infty^2-\mu c^2 F(s_0)^2}{E_\infty+ (E_\infty^2-\mu c^2 F(s_0)^2)^{\frac 12}  \sin (\theta_0 + 2 \mu^{\frac 12} a)}} \cos (\theta_0 + 2 \mu^{\frac 12} a)
+O(a^{-\frac 32}),\\
&\sigma (a)= -\frac {cF}{a\Gamma} = - \sqrt \mu c F(s_0)  a^{-\frac 12}  \Big(E_\infty+ (E_\infty^2-\mu c^2 F(s_0)^2)^{\frac 12}  \sin (\theta_0 + 2 \mu^{\frac 12} a)\Big)^{-\frac 12} 
+O(a^{-\frac 32}).
\end{aligned}\]
\end{prop}

\begin{rem}
\begin{enumerate}
\item According to Lemma \ref{L:+}, all solutions $W(a)$ of \eqref{eq4.8} given by Proposition \ref{P:local} satisfies the assumptions of the above proposition. 
\item The first scenario in the above is a rare phenomenon. In fact, one can prove that, among all solutions satisfying the assumptions (near $a=\infty$) of the above proposition, exactly one solution is in the first scenario. However, the proof is more technical and also it is not clear whether that particular solution $W(a)$ satisfies $s(0) = 0$. So we omit that proof.  
\item Recall $(s, \beta)$ is the polar coordinate system on the target surface. The angular derivative 
\[
\beta_a = \frac \sigma\Gamma = -\frac {c F(s_0)}{w^2}+ O(\frac 1a) \ge C_0>0
\] 
for some constant $C_0>0$. Therefore, unlike the case of $\mu<0$ where $W(a)$ converges to the rotation center $s=0$ (at the tangential level) along one single geodesic as $a\to \infty$, here the curve $W(a)$ converges to the other rotation center $s=s_0$ in a spiraling fashion with the angular speed uniformly bounded away from $0$. 
\end{enumerate}
\end{rem}

\subsection{Equivariant standing waves of the Ishimori system} \label{SS:VWeak}

In this subsection, given $(\mu,k,c)$ none of which vanishes, we construct weak solutions $\big(u(t, x, y), \phi(t, x, y)\big)$  of the original Ishimori system \eqref{eq1.1} targeted on $\mathbb{S}^2$ or $\mathbb{H}^2$ which are $(\mu,k,c)$-equivariant in the sense of \eqref{equiva}. In hyperbolic coordinates in each cone of $\R^2$, these solutions are written in terms of $\big(W(a), \psi(a)\big)$ as in \eqref{eq4.1.1} and satisfy equations \eqref{eq4.2}--\eqref{eq4.4} in respective cones. \\

\noindent {\it The case of $\mathbb{S}^2$.} 
Take a constant $b$ such that 
\[
0< c^2 < - 4 (b k + k^2).
\]
Let 
\[
\big( W_\pm^{h,v}(a), \psi_\pm^{h,v} (a)\big), \quad a>0,
\] 
be solutions of \eqref{eq4.2} with parameters $(\mu, k, c, b)$ (for solutions with supscript `h') and $(-\mu, k, c, b)$ (for solutions with supscript `v').
such that
\[
W_\pm^{h,v} (0) =\text{north pole} = (1, 0, 0), \quad \lim_{a\to 0^+} a|\p_a W_\pm^{h,v} (a) |=0. 
\] 
The existence of such solutions are guaranteed by \eqref{si} and Proposition \ref{P:local}, \ref{P:-}, and \ref{P:+}. The functions $\psi_\pm^{h,v}$ have logarithmic growth as $a \to 0^+$ and $W_\pm^{h,v}(a)$  are locally H\"older continuous. Moreover, 
\[
|\p_a W_\pm^{h,v} (a)| = O(a^{\kappa-1}), \qquad  \kappa = \sqrt{-(k^2 + b k)-(\frac{c}{2})^2}>0.
\]

As $a\to +\infty$, depending on the sign of $\mu$, two of the four solutions $W_\pm^{h,v}(a)$ converge to the north pole and the other two converge to the south pole. As in the previous subsections, denote the polar coordinates on $\mathbb{S}^2$
\[
s_\pm^{h,v} (a) = s_\pm^{h,v} \big( W_\pm^{h,v} (a)\big)= O(a^{\kappa }), \quad \beta_\pm^{h,v} (a) = \beta_\pm^{h,v}  \big( W_\pm^{h,v} (a)\big).
\]
From Lemma \ref{L:sigma}, 
\[
\p_a\beta_\pm^{h,v} = \sigma_\pm^{h,v}, \text{ where } \sigma_\pm^{h,v} (a) = - \frac ca \tan \big(\frac12 s_\pm^{h,v} (a) \big).
\]
Therefore $\beta_\pm^{h,v} (a)$ converges to some $\beta_\pm^{h,v} (0)$ at the rate $O(a^{\kappa_\pm^{h,v}})$ as $a \to0+$. Due to the rotational invariance, without loss of generality, we may require 
\[
\beta_\pm^{h,v}(0)=0.
\]

With $\mathbb{S}^2$ embedded into $\R^3$, we construct solutions of \eqref{eq1.1} by giving their values in each cone $C_\pm^{h,v}$ 
as	
\begin{equation} \label{weakU}\begin{split} 
u(t, x, y) =(u_{0\pm}^{h,v}, u_{1\pm}^{h,v}, u_{2\pm}^{h,v}) =& \cos \big(s_\pm^{h,v} (a)\big) \mathbf{\vec{e}_0}\\
&+  \sin\big(s_\pm^{h,v} (a)\big) \cos \big( \beta_\pm^{h,v}(a) + k\alpha+ \mu t\big)\mathbf{\vec{e}_1} \\
&+ \sin\big(s_\pm^{h,v} (a)\big) \sin \big( \beta_\pm^{h,v}(a) + k\alpha+ \mu t\big)\mathbf{\vec{e}_2}\\
\end{split} \end{equation}
where 
\[
a= \sqrt{|x^2 - y^2|}=\sqrt{|\xi\eta|} \qquad \alpha = \frac 12 \log \frac {|x+y|}{|x-y|} =\frac 12 (\log|\xi| - \log |\eta|). \]
Here as in Section \ref{weak}, 
\[
\xi=x+y, \qquad \eta = x-y.
\]
From equation \eqref{si}, the field $\phi_\pm^{h,v}$ can be chosen to be 
\begin{equation} \label{weakphi}
\phi_\pm^{h,v}  (t, x, y) =  b \log{ a }+ c\alpha - \int_0^a\frac{2k}{a'}\big(1-\cos (s_\pm^{h,v} (a')\big)\ da' 
\end{equation} 
where $a$ and $\alpha$ are given as in the above. 

From our construction, $u(t, \tau, \pm\tau) = (1, 0, 0)$ for all $\tau \in \R$ and functions $\big(u, \phi\big)(t, x, y)$ satisfy \eqref{eq1.1} pointwisely at any $(t, x, y)$ as long as $|x| \ne |y|$. 
In order to show that they are weak solutions of \eqref{eq1.1}, we will apply Propostion \ref{P:compatibility}. We start with verifying conditions \eqref{eq2.1}. Since $u$ is continuous and $\phi$ has at most logarithmic singularity (as $\xi, \eta \to 0$), it is obvious that $u, \phi \in L_{loc}^1$. The estimate on $u_\xi$ and $u_\eta$ are similar and we shall focus on $u_\xi$. Being tangent vectors on $\mathbb{S}^2$, one may compute $u_\xi$ using \eqref{eq4.1.1} 

\begin{equation} \label{temp1}\begin{split}
|u_\xi| =& |a_\xi W_a|+|k \alpha_\xi \Gamma (s) \CJ\p_s | \le \frac 12 \sqrt{\frac {|\eta|}{|\xi|}} (|s_a| + |\sigma|) + \frac{ |k\sin s(a)|}{2|\xi|}\\
\le & O(|\eta|^{\frac 12} |\xi|^{-\frac 12} a^{\kappa-1} + a^{\kappa} |\xi|^{-1} ) \le O(|\eta|^{\frac \kappa2} |\xi|^{\frac \kappa 2 -1}) 
\end{split}\end{equation}
This inequality, along with a similar estimate on $u_\eta$, implies that $\nabla u$, $\phi \nabla u$, and $u_x \wedge u_y$ belong to $L_{loc}^1(\R^2)$ and thus \eqref{eq2.1} is satisfied. 

In order to verify \eqref{weakSAssump}, we rewrite the above defined $\phi$ near the cross $\xi\eta=0$ to single out the logarithmic terms
\[
\phi_\pm^{h,v}  (t, \xi,\eta) = \phi_{1\pm}^{h,v}  (t, \xi) \log |\eta| + \phi_{2\pm}^{h,v}  (t, \eta) \log|\xi| + \phi_{3\pm}^{h,v}  (t, \xi,\eta).
\]
Here 
\[\begin{split}
&\phi_{1\pm}^{h,v} (t, \xi) = \frac 12 b - \frac 12 c, \quad \phi_{2\pm}^{h,v} (t, \eta) = \frac 12 b +\frac 12 c\\
&\phi_{3\pm}^{h,v} (t,\xi, \eta) = - \int_0^a\frac{2k}{a'}\big(1-\cos (s_\pm^{h,v} (a')\big)\ da' = O(a^{2\kappa}) = O\big((|\xi||\eta|)^{\kappa}\big).
\end{split}\]
Obviouly $\phi_{(1,2,3)\pm}^{h,v}$ satisfy all the conditions given in \eqref{weakSAssump}. Since $u(t, \xi, 0) \equiv (1, 0, 0)$, clearly $|u_\xi (t, \xi, 0)|=0$. From the above estimate on $u_\xi$, take any $1<p$ such that $(1- \frac 12 \kappa) p<1$ which is always possible, then on any finite interval $I'$ of $\xi$, we have 
\[
|u_\xi(t, \cdot, \eta)|_{L_\xi^p (I')} = O( |\eta|^{\frac {\kappa}2}).
\]
Similar estimate holds for $u_\eta$ and thus conditions \eqref{weakSAssump} are satisfied. 

Finally, we look at the conditions in Proposition \ref{P:compatibility}. Since $\phi_{1,2}$ are $t$-independent and piecewise constant, $u_\xi(t, \xi, 0)\equiv 0 \equiv u_\eta(t, 0, \eta)$, and $\phi_{3\pm}^{h,v}(t, \xi, 0) \equiv 0 \equiv \phi_{3\pm}^{h,v} (t, 0, \eta)$ which implies $\phi_3$ is continuous on $\R^2$, we only need to consider the continuity of $\phi_{1,2}^\pm$ which in our case is equivalent to 
\[
\phi_{1+}^v = \phi_{1-}^h, \quad \phi_{1+}^h =\phi_{1-}^v, \qquad \phi_{2-}^h= \phi_{2-}^v, \quad \phi_{2+}^h = \phi_{2+}^v \qquad \text{at } (0,0). 
\]
These are satisfied as these constant functions are indentical. 
Therefore the above constructed $(u, \phi)$ are weak solutions of \eqref{eq1.1} which are equivariant. \\

\noindent {\it The case of $\mathbb{H}^2$.} 
The construction of weak solutions of \eqref{eq1.1} which are weak solutions targeted on $\mathbb{H}^2$ is very similar with \eqref{weakU} and \eqref{weakphi} slightly modified. We will not repeat the whole procedure. Compared to the $\mathbb{S}^2$ case, the major difference in the $\mathbb{H}^2$ case is that we have to take the value of $u$ in one pair of cones (either horizontal or vertical) to be constant $(u_0=1, u_1=0, u_2=0)$ since we did not obtain Proposition \ref{P:+} for noncompact target manifolds. Namely, if $\mu>0$, $u$ takes constant value in horizontal cones and if $\mu<0$, $u$ takes constant value in vertical cones. The verification that these $(u, \phi)$ are weak solutions also follows from the same procedure. In particular, the same inequality \eqref{temp1} yields the local (in $\xi$ and $\eta$) estimate on the geometric length $(\langle  u_\xi,  u_\xi \rangle_-)^{\frac 12}$ of $u_\xi$ on $T_u\mathbb{H}^2$. Since on any 
bounded subset of $\mathbb{H}^2$, when  restricted to $T\mathbb{H}^2$, the Lorenz metric is equivalent to the Euclidean metric, we conclude $\nabla u \in L_{loc}^1$. Other conditions of Proposition \ref{P:compatibility} are verified following identically same steps. 

The proof of Main Theorems \ref{M1} and \ref{M2} is completed. 

\subsection{Radial standing waves.} \label{SS:radial}

The radial solutions can be thought as a special class of equivariant solutions, which correspond to $k=0$. The main statement of this subsection is the following lemma. 

\begin{lemma} \label{L:radial}  
No nontrivial radial standing wave of \eqref{eq1.1} satisfies \eqref{AssumAtZero}. 
\end{lemma}

\begin{proof} 
We will prove the lemma by contradiction. We will only consider \eqref{eq4.2} in the right side horizontal cone. Since $k=0$, equation \eqref{eq4.8} becomes 
\begin{equation} \label{eq4.8.0}
D_aW_a+\frac 1a W_a -\mu \nabla F+  \frac ca \mathcal{J} W_a=0.
\end{equation}
Meanwhile through the same argument as in Lemma \ref{L:sigma} and using \eqref{AssumAtZero}, one may verify that 
\begin{equation} \label{sigma1}  
\p_a \big(a \Gamma \sigma + c F \big) =0 \; \implies \; \sigma = \frac {cF}{a \Gamma}.   
\end{equation}
We will consider the cases of $c\ne0$ and $c=0$ separately.\\

\noindent {\it The case of $c\ne 0$.}
Multiplying it by $a^2 W_a$ we obtain 
\[
\p_a\big( \frac 12 a^2 |W_a^2| \big) = \mu a^2 \Gamma s_a.  
\]
This and the above formula for $\sigma$ implies 
\[
|\p_a\big( a^2 s_a^2 +\frac {c^2F^2}{\Gamma^2}  \big)| = 2|\mu a^2 \Gamma s_a| \le \frac {|\mu|}{|c|} (a^2 s_a^2 + c^2 a^2 \Gamma^2).
\]
For $|a|<<1$ which yields $|s|<<1$, we obtain from $\frac F\Gamma = O(\Gamma)$, \eqref{AssumAtZero}, and Gronwall inequality 
\[
a^2 s_a^2 +\frac {c^2F^2}{\Gamma^2} \equiv 0 \implies s(a) \equiv 0. 
\]

\noindent {\it The case $c=0$.} 
In this case, \eqref{eq4.8.0} and \eqref{sigma1} imply 
\[
s_{aa} + \frac 1a s_a -\mu \Gamma (s) =0.
\]
Let $\rho = a s_a$ and $\tau = \log a$ 
and it is easy to obtain 
\begin{equation} \label{eq4.49}\left\{\begin{aligned}
s_\tau =& \rho \\
\rho_\tau = &  \mu a^2 \Gamma \\
a_\tau =&a.
\end{aligned}\right. \end{equation}
Linearize at $(s,\rho,a)=(0,0,0)$, the Jacobian matrix is
$\begin{pmatrix}
0 & 1 & 0\\
0 & 0 & 0\\
0 & 0 & 1
\end{pmatrix}$.
The 2-dim center manifold corresponds to $\{ a=0\}$ and the 1-dim unstable manifold is exactly $\{s=0, \; \rho=0\}$ and thus no nontrivial solution exists satisfying \eqref{AssumAtZero}.
\end{proof} 

\section{Equivariant standing waves in similarity variables.}\label{self-similar}

We note that \eqref{eq1.1} is invariant under scaling
\[u_\lambda(t,x,y)\triangleq u(\lambda^2t, \lambda x,\lambda y)\ , \ \phi_\lambda(x,y)\triangleq\phi(\lambda x,\lambda y).\]
Thus, we also consider equivariant standing wave solutions of \eqref{eq1.1} in the similarity variables
\[
\big(u, \phi\big) (t, r, \alpha) = \big(R(\mu \log{|t|}+k\alpha)W( r ),\psi(r)+c\alpha\big),
\]
where $r=a|t|^{-\frac 12}$ is the similarity variable and $a$ is the hyperbolic radial variable. Without loss of generality, we take $t>0$, namely, $r=at^{-\frac 12}$. 
Our main result of this section is 

\begin{mainthm}\label{M3}
For any given $(\mu,k,c)\in\R^3$ none of which vanishes, system \eqref{eq1.1} admits nontrivial $(\mu,k,c)$-equivariant standing wave solution solutions from $\R^+ \times \R^2$ to $\mathbb S^2$ or $\mathbb H^2$ which are weak solutions. These solutions satisfy $u(t, \tau, \pm \tau)=(1,0,0)$ for all $\tau \in \R$ and take the following form in each of the four cones on $\R^2$ 
\[\begin{aligned}
u_\pm^{h,v} (t,x,y)& =\big(u_{0\pm}^{h,v}, u_{1\pm}^{h,v}, u_{2\pm}^{h,v}\big)\\
&=\big(\Gamma_s(s_\pm^{h, v} (|r|)), \Gamma(s_\pm^{h,v} (|r|)) \cos{(\beta_\pm^{h, v} (|r|)+ k\alpha+\mu\log t)}, \\
&\hspace{4.2cm}\Gamma(s_\pm^{h, v}(|r|))\sin{(\beta_\pm^{h,v} (|r|)+ k \alpha+\mu\log t)}\big),\\
\phi_\pm^{h,v}(t,x,y)&=c\alpha + b \log |r| - \int_0^{|r|}\frac{2k}{r'}F(s(r'))\ dr',
\end{aligned}\]
where $\Gamma$ is given by \eqref{Gamma} and 
\[
|r| = \sqrt{\frac 1t |x^2 - y^2|}, \quad \alpha = \frac 12 \log \frac {|x+y|}{|x-y|}, 
\]
$b$ is any constant satisfying $-4(k^2 + bk) > c^2$ and $F(s) = \int_0^s \Gamma(s') ds'$.  Moreover, 
\[
s_\pm^{h, v} (0)=0, \qquad s_-^{h, v} (r) = s_+^{h, v} (r), \quad \beta_+^{h, v} (r) = \beta_-^{h, v} (r), \quad \forall r\ge 0
\]  
and $s_\pm^{h, v} (r)$ and $\beta_\pm^{h,v} (r)$ are smooth for $r>0$ and are locally $C^\kappa$ for $r\ge 0$ with the exponent $\kappa = \sqrt{-(k^2+bk)-\frac{c^2}{4}}$. Finally, as $r \to +\infty$
\[
s_{\star\pm}^{h, v}  \triangleq \lim_{r\to +\infty} s_{\pm}^{h, v} (r) >0, \quad s_\pm^{h, v} (r) = s_{\star\pm}^{h, v} + O(r^{-2}), \quad \p_r u_\pm^{h, v} = O(r^{-1}).
\]
\end{mainthm}

\begin{rem} 
In fact, in the case of $\mathbb{S}^2$, if 
\[
\frac {c^2}4 + 3 k^2 + 2 \mu c <0
\]
then it is possible to choose $b$ such that in each cones the limit $s_\star \in (0, \pi)$. This is due to Corollary \ref{C:self}. The above condition allows us to choose $b$ such that the conditions there are satisfied. 
\end{rem} 

We notice that the existence of equivariant standing waves in similarity variables is very similar to that in regular hyperbolic variables, as well as their properties near the cross $|x|=|y|$. However, their asymptotic properties at $|r|= +\infty$ and $|a|= +\infty$ are drastically different. Equivariant standing waves in regular hyperbolic variables converge to one of the rotation centers on the target manifold, while those in similarity variables converge to points on the target manifolds which are most likely not rotation centers. 

\subsection{Existence of solutions} \label{SS:existence}

Equations \eqref{eq3.3}--\eqref{eq3.5} can written as
\begin{equation}\label{eq5.1}
C^h:\left\{\begin{aligned} &-\frac r2 W_r+\mu\p_\beta=\mathcal{J}(D_rW_r+\frac
1r W_r-\frac{k^2}{r^2}D_{\p_\beta} \p_\beta)+\frac{1}{r}(k\psi_r\p_\beta - cW_r),\\
&\psi_{rr}+\frac 1r\psi_r=-\frac {2k\Gamma(s) }{r} s_r,
\end{aligned}\right.
\end{equation}

\begin{equation} \label{eq5.3}
C^v:\left\{\begin{aligned}
&-\frac r2 W_r+\mu\p_\beta=-\mathcal{J}(D_rW_r+\frac 1r
W_r-\frac{k^2}{r^2}D_{\p_\beta} \p_\beta)
-\frac{1}{r}(k\psi_r\p_\beta - cW_r),\\
&\psi_{rr}+\frac 1r\psi_r=-\frac {2k\Gamma(s)}{r} s_r,
\end{aligned}\right.
\end{equation}

It is clear that these equations are invariant under $r \to -r$, so we will consider $r>0$ in the following. 

Comparing the above systems with \eqref{eq4.2} and \eqref{eq4.4}, we observe that the only difference is the extra term $-\frac r2 W_r$. On the one hand, this term does make substantial difference in our analysis of the asymptotic properties of these ODE systems. Indeed the solutions of the ODE systems may not converge to rotation centers as $a \to \infty$ and we will obtain less detailed asymptotic expansion of solutions as $r \to \infty$. Moreover, one also realizes that the first part of symmetry properties in Remark \ref{rem5.1} doesn't hold here. On the other hand, some part of the analysis does not depend on the sign in front of $\frac r2 W_r$ and are similar to that of \eqref{eq5.1} and \eqref{eq5.3}. Therefore, we will still mainly focus on equation \eqref{eq5.1} in the horizontal cones. Some details will be skipped because some of the analysis we are able to carry out in this section (especially for local solutions where $r$ is small) is much as (but somewhat less than) those in the regular hyperbolic variable case in Section \ref{eq}.

As in Section \ref{eq}, we will work under assumptions 
\begin{align} 
&b = \displaystyle\lim_{r\to0} r\p_r\psi_\pm^{h, v} (a) \text{ exists and are identical for the four cones.} 
\label{H2}\\
&\lim_{r\to 0\pm} r W_r (r) =0 \quad \text{ and } \quad s(0)=0  \label{SAssumAtZero} 
\end{align} 
In particular, as in Section \ref{eq}, $b$ is chosen such that 
\begin{equation} \label{self:b}
c^2 < - 4(kb + k^2).
\end{equation}
From \eqref{eq5.1}, we have
\begin{equation} \label{eq5.5}
D_rW_r+\frac 1r W_r + \frac 1{r^2} \nabla G -\mu \nabla F+  (\frac cr-\frac r2) \mathcal{J} W_r=0,
\end{equation}
where $F$ and $G$ are defined as in \eqref{eq4.8}. It is easy to see Lemma~\ref{L1} still holds in this case. Therefore, we continue to 
work under the condition $k^2+b k<0$ in the rest of this section.

Projecting this equation to the directions of $\p_s$ and $\mathcal J\p_s$, we obtain
\begin{eqnarray}\label{eq5.6}
s_{rr} &=&\frac {\Gamma_s}\Gamma \sigma^2 - \frac 1r s_r - \frac 1{r^2} G_s + \mu \Gamma + (\frac cr-\frac r2) \sigma,\\\label{eq5.7}
\sigma_r &=& - \frac {\Gamma_s}\Gamma \sigma s_r - \frac 1r \sigma +(\frac r2- \frac cr) s_r,
\end{eqnarray}
where $\sigma=\langle W_r, \mathcal J\p_s\rangle$. Let
\[\frac{rs_r}{\Gamma}=\rho\cos\gamma\ , \ \frac{r\sigma}{\Gamma}=\rho\sin\gamma.\]
By \eqref{eq5.6} and \eqref{eq5.7} and let $p=\log r$, we have
\begin{equation} \label{eq5.8}\left\{\begin{aligned}
s_p =& \rho \Gamma \cos \gamma\\
\rho_p = &( - \Gamma_s \rho^2 - \frac {G_s}\Gamma + \mu r^2) \cos \gamma\\
\gamma_p  =& ( -\Gamma_s \rho +  \frac {G_s}{\rho\Gamma} - \mu \frac {r^2} \rho ) \sin \gamma - c+\frac{r^2}{2}\\
r_p =&r.
\end{aligned}\right. \end{equation}

\begin{rem}
By applying the same analysis to \eqref{eq5.3}, one obtains a system similar to \eqref{eq5.8} except the last term $\frac{r^2}{2}$ in $\gamma$-equation is replaced by $-\frac{r^2}{2}$.
\end{rem}

It is clear that the method in Subsection \ref{SS:local} is also applicable to  \eqref{eq5.8} because the term $\frac{r^2}{2}$ vanishes up to the first order at $r=0$. Therefore, we skip the proof of the following proposition.

\begin{prop} \label{self:local}
Suppose $c^2<-4(k^2+bk)$. There exist $r_\star>0$ and $s_* >0$ such that for any $\tilde s \in [0, s_*)$, there exists a unique solution $(s, \rho, \gamma) (r)$ of \eqref{eq5.8}, or equivalently a unique solution $W(a)$ of \eqref{eq5.5} up to the rotation, such that its domain contains $[0, r_\star]$ and $s(r_\star) = \tilde s$ and $s(0) =0$. Moreover, this solution satisfies $s(r) = O(r^\kappa)$ and $|W_r(r)| = O(r^{\kappa-1})|$ for $|r|<<1$, where $\kappa = \sqrt{-(k^2+bk)-\frac{c^2}{4}}>0$. 
\end{prop}

\begin{rem}
\begin{enumerate}
\item The solution $(s,\rho,\gamma)$ obtained above can be used to construct solutions of \eqref{eq5.1} (in horizontal cones). The above proposition also applies to yield solutions of \eqref{eq5.3} (in vertical cones) as $\pm\frac {r^2}2$ is a higher order term for $|r|<<1$ and does not affect the argument. 
\item Under same type of assumptions, the global existence of solutions of \eqref{eq5.5} is obtained in exactly the same way as in Lemma \ref{L:global} where the term involving $\CJ W_a$ did not play any role.
\item Following the same arguments as in Subsection \ref{SS:VWeak}, one may verify that solutions to \eqref{eq5.1} and \eqref{eq5.3} yield weak solutions of \eqref{eq1.1} for $t\ne 0$.  
\end{enumerate}
\end{rem}

The above proposition and remarks completing the proof of the statements in the main theorems concerning the existence of solutions and their properties for $|r|<<1$. We shall focus on analyzing the asymptotic behavior of these solutions as $r \to \infty$.

\subsection{Asymptotic properties of solutions of \eqref{eq5.1} at $r=+\infty$.}\label{self:global}

We first claim 

\begin{lemma} \label{self:global} 
Assume $c^2 < -4(bk + k^2)$.  
\begin{enumerate}
\item Suppose the target surface is compact then solutions obtained in Proposition \ref{self:local} exist for all $r$ and 
\begin{equation}\label{eq5.10}
\lim_{r\to\infty}\frac 12|W_r|^2-\mu F \text{ exists and } \quad \int_1^\infty\frac 1r |W_r|^2\ dr<\infty.
\end{equation}
\item Suppose $\mu<0$. For any $r_0>0$, there exists $\delta >0$ such that  if a solution obtained in Proposition \ref{self:local} satisfies $\frac 12|W_r (r_0)|^2-\mu F(s(r_0)) < \delta$ then it is defined for all $r>0$ and \eqref{eq5.10} holds.  
\end{enumerate}
\end{lemma}

\begin{proof}
The proof of the second part of the  lemma is same as as the that of the second part of Lemma \ref{L:global}, where the extra term $\frac r2 \CJ W_r$ disappears in the energy estimates. We omit the details. The proof of the first part where the target manifold is compact is only a small modification as well. To see this,  
we multiply \eqref{eq5.5} by $W_r$ to obtain
\begin{equation} \label{eq5.10.1} 
\partial_r(\frac12|W_r|^2-\mu F)=-\frac1r|W_r|^2 - \frac  1{r^2} G_s s_r  
\end{equation}
which yields
\[
\partial_r(\frac12|W_r|^2-\mu F) -  \frac 1{2r^3} G_s^2 = - \frac 1{r} |W_r|^2 + \frac 1{2r} s_r^2 - \frac 12 (r^{-\frac 12} s_r + r^{-\frac 32} G_s)^2 \le  - \frac 1{2r} |W_r|^2.
\]
Since $\Gamma$, $\Gamma_s$, and $F$ are uniformly bounded, we obtain that $ \frac 1{2r^3} G_s^2$ is integrable as $r \to \infty$. and 
\[
\frac12|W_r|^2-\mu F-\int_1^r  \frac 1{(r')^3} |G_s (s(r'))|^2 dr' 
\]
is decreasing and bounded from below. Therefore, this quantity converges and $\frac 1r |W_r|^2$ is integrable, which also implies the convergence of $\frac12|W_r|^2-\mu F$.
\end{proof}

\begin{lemma} \label{self:limit1}
Under the assumptions of Lemma \ref{self:global}, solutions obtained in Proposition \ref{self:local} satisfy 
\begin{enumerate} 
\item $s(r)\ne 0$ for any $r > 0$;
\item $|s_r|+|\sigma|=O(\frac 1r)$ as $r\to +\infty$;
\item $s_\star \triangleq \lim_{r \to +\infty} s(r)$ exists.
\end{enumerate} 
\end{lemma} 

\begin{proof}
From \eqref{eq5.7}, we have
\begin{equation}\label{eq5.12}
(r\Gamma\sigma)_r=(\frac{r^2}{2}-c)F_r.
\end{equation}
Multiplying \eqref{eq5.5} by $r^2W_r$ and using the above equality, 
we obtain
\[\p_r(\frac 12 |rW_r|^2+G-2\mu(r\Gamma\sigma+cF))=0.\]
Since $s(0)=0$ for all solutions obtained in Proposition \ref{self:local}, it implies 
\[
\frac 12|rW_r|^2+G-2\mu(r\Gamma\sigma+cF) =0,
\]
which is equivalent to
\begin{equation}\label{eq5.14}
r^2\left(s_r^2+(\sigma-\frac{2\mu\Gamma}{r})^2\right)=-2G+4\mu cF+4\mu^2\Gamma^2\triangleq H(s).
\end{equation}

The conclusions of the lemma follow immediately from the above identity. In fact, under the assumptions of Lemma \ref{self:global}, the right hand side of \eqref{eq5.14} is uniformly bounded for all $r$ (any $\mu$ for compact target surface and $\mu<0$ and small $s(r_0)$ for noncompact surfaces), we first obtain the estimates on $s_r$ and $\sigma$.  Consequently, those estimates along with \eqref{eq5.10} imply the convergence of $s(r)$. Finally, suppose $s(r_0) =0$ for some $r_0\ne 0$, which means $W(s(r_0))$ is at the starting rotation center, then \eqref{eq5.14} and the fact $H(0)=0$ immediately implies $W_r(r_0)=0$. By the uniqueness of solutions to \eqref{eq5.1}, we obtain $W_r \equiv 0$ which is a contradiction. 
\end{proof} 

Equality \eqref{eq5.14} actually has stronger implications.  It is clear the function $H(s)$ defined in \eqref{eq5.14} is smooth, even in $s$, and satisfies $H(0)=0$ and $H_{ss}(0)>0$ due to the assumption $c^2 < -4 (bk + k^2)$. Let 
\begin{equation} \label{H1}
s_1 = \sup \{s>0 \mid H(s') \ge 0 \text{ on } (0, s]\} >0.
\end{equation}
Clearly 
\[
s_1\in [0, s_0) \text{ or } s_1 = +\infty
\]
and Lemma \ref{self:limit1} and \eqref{eq5.14} imply  
\[
s(r) \in (0, s_1] \ \; \forall \ r>0 \text{ and } s_* \in[0, s_1] .  
\]

In fact, we can prove

\begin{lemma} \label{self:limit3} 
Under the assumptions of Lemma \ref{self:limit1}, solutions obtained in Proposition \ref{self:local} satisfy
\[
0< s_\star \le \min\{s_0, s_1\}  \quad \text{ and } \quad s(r) = s_\star + O(r^{-2}).
\]
\end{lemma}

\begin{proof} 
We will consider 3 different cases. \\

\noindent {\bf Case I: $s_\star =s_1\in (0, s_0)$.}  In this case, we only need to estimate $s-s_\star$. Clearly \eqref{eq5.14} implies $|W_r|=O(r^{-1})$ for $r>>1$. Integrating \eqref{eq5.10.1} from $r$ to $+\infty$, we obtain
\[
\mu \big( F - F(s_1)\big) + O(r^{-2}) = \int_r^{+\infty} O(r'^{-3}) \ dr' \implies F- F(s_1) = O(r^{-2}).   
\]
Since $s_1 \in (0, s_0)$ which implies  $F_s(s_1) \ne 0$ and thus 
\[
|s -s_\star| = |s- s_1| = O(r^{-2}). \\
\]

\noindent {\bf Case II: $s_\star =s_0 <s_1 =+\infty$ and $H(s_0) >0$.} Again we only need to estimate $|s-s_0|$ in this case. Near $s=s_0$, there exists a smooth function 
\[
\tilde s= \tilde s(s) \quad \text{ such that  } \quad \frac 12 \tilde s^2 = F(s_0) - F(s) \quad \text{ and }\quad \tilde s'(s_0) =-1.
\]
In this case, we will use $\tilde s$ to replace $s$ as the equivalent variable. One may compute using \eqref{eq5.14} and Lemma \ref{self:limit1}
\[
|W_r|^2 =  s_r^2 + \sigma^2 = r^{-2} H + 4 \mu r^{-2} \Gamma (r\sigma - \mu \Gamma) = r^{-2} H (s_0) + O(r^{-2} \tilde s).
\]
Equation \eqref{eq5.10.1} implies
\[
\p_r \Big( \frac 12 |W_r|^2  + \frac \mu2 \tilde s^2 + r^{-2} \big( G - G(s_0) \big) \Big) = - \frac 1r |W_r|^2 -\frac {2}{r^3} \big( G - G(s_0)\big). 
\]
Subtracting $- H (s_0) r^{-3}$ and dividing by $\mu$ in the above, integrating it from $r$ to $+\infty$, 
and singling out the principle terms, we obtain 
\[
\frac 12 \tilde s^2 + O(r^{-2} \tilde s) = \int_r^{+\infty} O(r'^{-3} \tilde s) \ dr'. 
\]
For $r>>1$, we have 
\[
\tilde s^2 \le \int_r^{+\infty} \frac 1{r'} \tilde s^2 \ dr' + O(r^{-4}).
\]
The next step is exactly same as in the proof of the Gronwall inequality. Namely, let 
\[
f(r) =  \int_r^{+\infty} \frac 1{r'} \tilde s^2 \ dr' \ge 0
\]
then it satisfies $f(+\infty) =0$ and 
\[
- r f' \le f + O(r^{-4}) \implies - (rf)' \le O(r^{-4}). 
\]
Therefore,
\[
0\le rf (r) \le O(r^{-3}),
\]
which implies
\[
\tilde s^2 \le f(r) + O(r^{-4}) = O(r^{-4}) \implies s_\star - s =s_0 - s= O(\tilde s)= O(r^{-2}). \\
\]

\noindent {\bf Case III: Others.} There exists an odd function $h(s)\in C^0( [-s_1, s_1])$ such that $h^2 = H$. Moreover, $h$ is smooth at any $s\in [-s_1, s_1]$ except possibly at $s_1$. In fact, $h$ is smooth at $s_1$ unless $H_s(s_1) < 0$, in which case  $h$ is H\"older with exponent $\frac 12$. In particular, since $H_s(0)=H_s(s_0)=0$, $h$ is smooth near $0$ and $s_0$ (if $s_0 \le s_1$). 

For $s(r)\notin \{ H\Gamma = 0\}$, equation \eqref{eq5.14} motivates us to consider
\begin{equation} \label{theta0}
s_r=\frac 1r h(s)\cos\theta,\qquad \sigma = \frac {2\mu\Gamma}r + \frac 1r h(s)\sin\theta. 
\end{equation} 
Using \eqref{eq5.12} one may compute 
\begin{equation} \label{theta1}
\theta_r = \frac r2 - \frac 1r \big(c + 4 \mu \Gamma_s + h_s \sin \theta\big) -\frac {\Gamma_s h \sin \theta}{r \Gamma} \triangleq \frac r2 + \frac 1r g(s, \sin \theta).
\end{equation} 
Even though the above equations are derived under the assumption $H\Gamma \ne 0$, it is easy to see  that $g$ can be extended smoothly onto $s \in [-s_1, s_1]$ except possibly 
\begin{enumerate} 
\item at $s_1$ if $s_1\in (0, s_0)$ and $H_s (s_1) <0$ or 
\item at $s_0$ if $s_0< s_1 = \infty$ and $H(s_0) >0$.
\end{enumerate}
In fact, if $s_1 =+\infty$ and $H(s_0)=0$ then $h$ is smooth near $s_0$ of order $O(|s_0-s_1|)$ and thus the  term $\frac {\Gamma_s h \sin \theta}{r \Gamma}$ is smooth in a neighborhood of $s_0$ as well. Therefore in this case
\[
g \text{ is smooth near } s_\star. 
\]

Obviously, \eqref{eq5.14} implies that $H(s(r))$ can not vanish on any nontrivial interval of $r$, otherwise $s_r(r)$ would be identically zero on such an interval and the uniqueness of ODE solutions implies that the solution $s(r)$ is a constant. Therefore, there exists $\delta, \ r_0 >0$ such that 
\[
g \text{ is smooth on } [s_\star-\delta, s_\star +\delta] \cap [-s_1, s_1]
\]
and 
\[
h(s(r_0)) \ne 0 \text{ and } s(r) \in [s_\star-\delta, s_\star +\delta] \cap [-s_1, s_1], \quad r \ge r_0.
\]
Let
\[
r_* = \sup \{r\ge r_0\mid h(s(r')) \ne 0, \; r' \in (r_0, r)\} > r_0. 
\]
For any $r_1, r_2 \in (r_0, r_*)$ with $r_1 \le r_2$, using the fact that the range of $s(r)$, $r\ge r_0$, is in a compact subset of the target surface which means continuous quantities such as $\Gamma$, $h$, {\it etc.} are effectively uniformly bounded, we may compute 
\[\begin{aligned}
&\log{|h(s)|}\Big|_{r_1}^{r_2} = \int_{r_1}^{r_2} \frac{h_s}{h}s_r\ dr =\int_{r_1}^{r_2} \frac{h_s\cos\theta\theta_r}{r\theta_r}\ dr\\
=&\frac{2h_s\sin\theta}{r^2 + 2g}\Big|_{r_1}^{r_2}-\int_{r_1}^{r_2} (\frac{2h_s}{ r^2 + 2g})_r\sin\theta\ dr\\
=&O(\frac {1}{r_1^2})+2\int_{r_1}^{r_2} \frac {2 h_s \big(r + \frac 1r \p_1 g h \cos \theta + \p_2 g (\frac r2 + \frac 1r g)\big) - (r + \frac 1r g) h_{ss} h \cos \theta}{(r^2 + 2g)^2}\sin\theta \ dr\\
=&O(\frac {1}{r_1^2})+\int_{r_1}^{r_2} O(\frac{1}{r^3})\ dr=O(\frac{1}{r_1^2})<\infty.
\end{aligned}\]
Therefore $r_* = +\infty$ and 
\[
\lim_{r \to +\infty} \log |h(s(r))| \text{ exists } \implies H(s_\star) \ne 0. 
\]
Immediately we obtain $s_\star \ne 0$. 

Finally, now that we have proved $s_\star \ne 0$ or $s_0$, we may compute 
\[
s(r_2) - s(r_1) = \int_{r_1}^{r_2} s_r \ dr = \int_{r_1}^{r_2} \frac  {h \cos \theta \theta_r}{r \theta_r} dr.
\]
Following the same integration by parts argument as in the above, we obtain $s(r) = s_\star+O(r^{-2})$. 

The analysis of the above 3 cases completes the proof.
\end{proof}

\begin{cor} \label{C:self} 
Suppose $s_0 <\infty$ and  
\[
-(bk + k^2) > \frac {c^2}4 \quad \text{ and } \quad 2\mu c - bk + k^2 F(s_0) <0  
\]
then $s_\star \in (0, s_0)$.
\end{cor} 

Obviously the first condition is for Proposition \ref{self:local}. The second condition is equivalent to $H(s_0)<0$ which implies $s_1 <s_0$ and thus $s_\star \le s_1 <s_0$.

\medskip

\section{Appendix: a basic statement of local invariant manifolds theorem}

In this appendix, we give a basic statement of the local invariant manifold theorem from dynamical system theory 
that were used in Subsections~\ref{SS:local} and \ref{SS:existence} to construct local solutions.
If readers are interested in more details or the proof, they can be found in, for example, \cite{Ch99, CL88}.

Let us consider a general nonlinear system near a fixed point 
\begin{equation}\label{eq6.2}
\begin{pmatrix}
\dot x\\ \dot y
\end{pmatrix}=\begin{pmatrix}
S & 0 \\
0 & U
\end{pmatrix}\begin{pmatrix}
x\\y
\end{pmatrix}+\begin{pmatrix}
f(x,y)\\g(x,y)
\end{pmatrix},\end{equation}
where $x\in X$, $y\in Y$, and $X$ and $Y$ are Banach spaces. Assume
\begin{enumerate}
\item[(A1)] $f \in C^k (X\times Y, X)$, $g\in C^k (X\times Y, Y)$, $k\ge 1$, and $f(0,0)=0$, $g(0,0)=0$, $Df(0,0)=0$, $Dg(0,0)=0$.
\item[(A2)] $S$ generates a $C^0$ semigroup $e^{tS}$ and $U$ generates a $C^0$ group $e^{tU}$.

\item[(A3)] There exist constants $K>0$ and $b>a$ such that 
\begin{equation}\label{eq6.3}
|e^{tS}|\leq Ke^{at}\ \text{ for } \ t\geq0\ \ ,\ \ |e^{tU}|\leq Ke^{bt}\ \text{ for } \ t\leq0.
\end{equation}
\end{enumerate}
In the following we will use $B_X(r)$ to denote the ball in $X$ centered at $0$ with radius $r>0$. 

\begin{theorem}
Consider \eqref{eq6.2} under assumptions (A1)--(A3). There exists $\delta >0$ such that the following properties hold. 
\begin{enumerate} 
\item There exists a $C^1$ surface 
\[
\mathcal W^s=\{(x,h_s(x))\big|x\in B_X(\delta), h_s \in C^k \big(B_X(\delta), B_Y (\delta)\big) \},
\]
which satisfies 
\begin{enumerate} 
\item $h_s(0)=0$ and $Dh_s(0)=0$. 
\item $h_s$ is $C^k$ if $b> ka$.
\item $\mathcal W^s$ is locally invariant, i.e., if $(x_0,y_0)\in\mathcal W^s$ and its trajectory 
$(x(t),y(t))\in B_X(\delta) \times B_Y(\delta)$ for all $t\in [0, T]$, then $(x(t),y(t)) \in \mathcal W^s$ for all $t \in [0, T]$. 
\item If $a<0$, there exist $\gamma\in \big(a,\min\{b, 0\}\big)$ and $K_1>0$ such that $(x_0, y_0) \in B_X(\delta) \times B_Y(\delta)$ belongs to $\mathcal{W}^s$ if and only if its trajectory $(x(t), y(t))$ satisfies 
\[
|x(t)| + |y(t)| \le K_1 ( |x_0| + |y_0|) e^{\gamma t}, \qquad t\ge 0. 
\]
This also implies the uniqueness of $\mathcal{W}^s$ when $a<0$.
\end{enumerate} 
\item There exists a $C^1$ surface 
\[
\mathcal W^u=\{(h_u(y), y)\big|y\in B_Y(\delta), h_u \in C^k \big(B_Y(\delta), B_X (\delta)\big) \},
\]
which satisfies 
\begin{enumerate} 
\item $h_u(0)=0$ and $Dh_u(0)=0$. 
\item $h_u$ is $C^k$ if $a< kb$.
\item $\mathcal W^u$ is locally invariant, i.e., if $(x_0,y_0)\in\mathcal W^s$ and its trajectory 
$(x(t),y(t))\in B_X(\delta) \times B_Y(\delta)$ for all $t\in [0, T]$, then $(x(t),y(t)) \in \mathcal W^u$ for all $t \in [0, T]$. 
\item For $(x_0, y_0) \in \mathcal{W}^u$, it has a backward trajectory and $\mathcal{W}^u$ is also locally invariant for these backward trajectories much as in property (c). 
\item If $b>0$, there exist $\gamma\in \big(\max\{a, 0\}, b\big)$ and $K_1>0$ such that $(x_0, y_0) \in \mathcal{W}^u$ if and only if it has a backward trajectory $(x(t), y(t))$, $t\le 0$, satisfying 
\[
|x(t)| + |y(t)| \le K_1 ( |x_0| + |y_0|) e^{\gamma t}, \qquad t\le 0.
\]
This also implies the uniqueness of $\mathcal{W}^u$ when $b>0$. 
\end{enumerate}\end{enumerate} 
\end{theorem}

\end{document}